\documentclass[12pt]{amsart}

\usepackage{amsmath,amssymb,amsthm}
\usepackage{graphics}
\usepackage{enumerate}
\usepackage[dvips]{color}
\usepackage{version}
\usepackage{stmaryrd}
\usepackage[all]{xy}

\usepackage{todonotes}

\newtheorem{Thm}{Theorem}[section] 
\newtheorem{Lem}[Thm]{Lemma}

\newtheorem{Cor}[Thm]{Corollary}
\newtheorem{Cor.Conj}[Thm]{Corollary of Conjecture}

\newtheorem{Prop}[Thm]{Proposition}

\newtheorem{Conj}[Thm]{Conjecture}

\newtheorem{Claim}[Thm]{Claim}

\theoremstyle{remark}
\newtheorem{Rem}[Thm]{Remark}
\newtheorem{Ex}[Thm]{Example}

\theoremstyle{definition}
\newtheorem{Def}[Thm]{Definition}

\newtheorem{Step}{Step}
\newtheorem{Stp}{Step}

\newtheorem{Aspect}{Aspect}

\newtheorem*{ack}{Acknowledgements}

\newtheorem{Case}{Case}
\newtheorem*{Subcase}{Subcase}

\newcommand{\coeff}{\mathop{\mathrm{coeff}}\nolimits}

\newcommand{\R}{\ensuremath{\mathbb{R}}}
\newcommand{\C}{\ensuremath{\mathbb{C}}}

\newcommand{\Z}{\ensuremath{\mathbb{Z}}}

\newcommand{\Q}{\mathbb{Q}}

\newcommand{\F}{\mathbb{F}}
\newcommand{\A}{\mathbb{A}}
\newcommand{\B}{\mathcal{B}}
\newcommand{\D}{\mathbb{D}}
\newcommand{\E}{\mathbb{E}}
\newcommand{\X}{\mathcal{X}}

\newcommand{\OO}{\mathcal{O}}
\newcommand{\M}{\mathcal{M}}
\newcommand{\PP}{\mathbb{P}}

\begin{document}

\title[For type II degenerating K3 surfaces]
{PL density invariant for \\ 
Type II degenerating K3 surfaces, 
Moduli compactification and hyperK\"ahler metrics}

\author{Yuji Odaka}
\date{\today}

\maketitle

\begin{abstract}

A protagonist here is a new-type invariant for type II degenerations of 
K3 surfaces, which is explicit 
{\it PL (piecewise linear) convex function} from 
the interval with at most $18$ non-linear points. 
Forgetting its actual function behaviour, it also classifies 
the type II degenerations into several combinatorial types, 
depending on the type of root lattices as appeared in 
classical examples. 

From differential geometric viewpoint, 
the function is obtained as the density function of the 
limit measure on 
the collapsing hyperK\"ahler metrics to conjectural segments, as in \cite{HSZ}. 
On the way, we also reconstruct a moduli compactification 
of elliptic K3 surfaces by \cite{Brun, AB, ABE} in a more elementary 
manner, 
analyze the cusps more explicitly. 

We also interpret the glued hyperK\"ahler fibration of \cite{HSVZ} as a special case from 
our viewpoint, discuss other cases, and 
possible relations with Landau-Ginzburg models in 
the mirror symmetry context. 

\end{abstract}


\setcounter{tocdepth}{1}
\tableofcontents

\section{Introduction}

In this paper, to each type II degeneration of polarized 
K3 surfaces $\pi\colon (\mathcal{X},\mathcal{L})\to \Delta=\{t\in \C\mid 
|t|<1\}$, we associate an explicit piecewise linear convex function $V=V_{\pi}\colon [0,1]\to \R_{\ge 0} (\cup \infty)$ 
over the interval, as a new type invariant and discuss its geometric meanings 
from various geometric perspectives. The non-differential points of  $V$ are 
at most $18$ points and anyhow the behaviour of $V$ 
is completely classified. 
If $\mathcal{L}$ are assumed to be of relative Hodge (integral) class as in algebraic geometry, the function $L$ is rational while 
if we extend to relative K\"ahler class on (not necessarily algebraic) 
$\X$, 
then we obtain not necessarily rational bend points. 

From differential geometric perspective, 
this is done by considering the behaviour of hyperK\"ahler metrics on 
the fibers $\mathcal{X}_{t}=\pi^{-1}(t)$ with the K\"ahler class 
in $\R_{>0} c_{1}(\mathcal{L}_{t}=\mathcal{L}|_{\mathcal{X}_{t}})$ 
with diameter bounded rescale, 
as our function $V$ is the density function of 
a limit measure on the conjectural limit interval as 
predicted in recent \cite{HSZ}. 
As inferred from such background, we can actually define $V$ 
for not only holomorphic one parameter degeneration but for 
more general sequences ``of type II''. 

The ends behaviour of $V$ is encoded in the 
root lattices of type D or E while the open part is reflected in 
type A lattices. This root lattice-theoritic information has 
classically appeared and studied at least in lower degree case 
e.g., in \cite{Fri}, and also in 
recent \cite[\S 3B, 9.10]{AET}, \cite[\S 1]{LO},  and \cite{ABE}. 
Our exploration aims to reveal their hidden meanings. 

\subsubsection*{History of this work}
This paper originally stems out as a part of the series for 
ongoing joint work with Y.Oshima on 
collapsing of hyperK\"ahler metrics, with recent focus on 
K3 surfaces to segments, with great inspirations input 
from \cite{HSZ} and \cite{ABE} as well. 
Our whole framework depends on the one 
initiated in our previous joint paper \cite{OO18} (its short 
summary is \cite{OO.announce}), 
whose particular focus of the latter part was on type III degenerations and associated collapsing 
to spheres. 

Also the recent log KSBA style 
explicit compactification work 
of moduli of {\it elliptic} K3 surfaces by \cite{Brun, ABE, AB} has much to do with 
our work. In particular, $V$ implicitly appears 
in \cite{ABE} in the form of their 
integral affine spheres construction, and used in the projective moduli variety 
construction, 
much to our surprize then. 

\vspace{2mm}
Here is the comparison, partially to give an overview of this paper. 

\subsubsection*{Comparison and Organization}
While \cite[\S 7A]{ABE} implicitly obtained the definition of $V$ 
in the form of its ``graphs'' as integral affine spheres, 
Oshima (\cite{Osh}) also 
had definition of $V$ indepedently, as a function 
for the collapsing K3 surfaces to segment. 
Then he proves that it is the limit measure of 
the McLean metric on $\PP^{1}$, 
from periods calculation along explicit 
$2$-cycles. 

Part 2 of this paper provides another algebro-geometric proof 
of the theorem of Oshima, which is the heart of the paper. 
Before that, in preparatory Part 1, we give an elementary reproof 
and analysis of the stable reduction 
corresponding to \cite{ABE}. The 
reproof has virtue for the arguments in Part 2. More precisely 
speaking, the information of asymptotic behaviour of 
singular fibers analyzed in Part 1, 
not only the location of limits of discriminants, is crucially 
used in Part 2. 

In Part 2, we also connect our work with \cite{HSVZ}, 
in which we interpret as a special case with the ``type EAE''. 
There are other 
interesting cases whose label include ``type D''.

\vspace{3mm}

Part 1 can and do work over an arbitary 
algebraically closed field $K$ of characteristic neither 
$2$ nor $3$, unless otherwise stated. The assumption on 
characteristic is frequently used, 
especially for the Weierstrass standard form description of 
elliptic curves and the reducedness of the finite 
group schemes $\mu_{2}$ and $\mu_{3}$ over $K$. 

On the other hand, Part 2 works over $\C$ as, 
for instance, discussions involve hyperK\"ahler metrics. 

\begin{ack}
As noted above, this paper stems out as a part of 
collaboration with Y.Oshima, and we plan for more sequels. 
So first of all, the author thanks Y.Oshima 
for the ongoing fruitful and enjoyable discussions, 
as well as the permission to emit this part of results in this form.  

We also would like to thank 
V.Alexeev, K.Ascher, P.Engel, S.Honda, 
H.Iritani, S.Sun, J.Viaclovsky, Y-S.Lin 
for the helpful and friendly discussions. 
The author is partially supported by 
KAKENHI 18K13389 (Grant-in-Aid for Early-Career Scientists), 
KAKENHI 16H06335 (Grant-in-Aid for Scientific Research (S)) and 
KAKENHI 20H00112 (Grant-in-Aid for Scientific Research (A)) 
during this research. 
\end{ack}
\part{Moduli of elliptic K3 surfaces revisited}
\section{Review of \cite[\S 7]{OO18} and analysis of cusps}

In the work \cite[\S7]{OO18} on collapsing of K3 surfaces, 
the moduli $M_{W}(\C)$ of complex 
Weierstrass elliptic K3 surfaces played an important role as it parametrizes  
real $2$-dimensional collapses (``tropical K3 surfaces'') of 
K\"ahler K3 surfaces. 
Still keeping it as one of the motivations, 
we first make further analysis on $M_{W}$ 
in this paper. It also naturally extends to 
other field $K$. First, we set up or recall the notation. 

We set $\mathbb{A}^{22}$, which 
parametrizes the 
coefficients of degree $8$ polynomial $g_{8}$ 
and the 
coefficients of degree $12$ polynomial $g_{12}$. 

Recall from \cite[\S7.1]{OO18} that $\overline{M_{W}}$ is nothing but the GIT quotient of 
$\mathbb{A}^{22}_{g_{8},g_{12}}\setminus \{0\}$ by the action of ${\rm GL}(2)$, 
or in other words, that of 
\begin{align}\label{WP.par}
\mathbb{P}(
\underbrace{2,2,2,2,2,2,2,2,2}_9,
\underbrace{3,3,3,3,3,3,3,3,3,3,3,3,3}_{13})\\
(=(\mathbb{A}^{22}_{g_{8},g_{12}}\setminus \{0\})/\mathbb{G}_{m}(K))
\end{align}
by the further action of $SL(2)$. 
We denote the homogeneous coordinates of the base $\mathbb{P}^{1}$ 
as $s_{1}$ and $s_{2}$, and set $s:=\frac{s_{1}}{s_{2}}$.

Recall that \cite[\S 7.2.1]{OO18} shows 
$\overline{M_{W}}$ is isomorphic to the Satake-Baily-Borel compactification for an appropriate $O(2,18)$ orthogonal 
symmetric variety (and also has the structure as 
a double ({\it anti-(!)}holomorphic) 
covering over the boundary component $\mathcal{M}_{\rm K3}(a)$ 
of $\overline{M_{\rm K3}}$ in \cite[\S6]{OO18}), which appears in the context of F-theory e.g., as {\it classical F-theory moduli space} in \cite{CM05}. 
See \cite[\S 6.1]{OO18}, 
in particular its last discussion for the 
proof of Theorem 6.6 of 
\S 7.3.7 in {\it loc.cit} for the details. 

\vspace{5mm} 
Now we head towards more explicit understanding of cusps of the 
compactification. 
From the uniformization 
structure $M_{W}\simeq \Gamma\backslash \mathcal{D}$, 
with orthogonal symmetric domain $\mathcal{D}$, 
there is 
the natural branch divisor $B$ in $M_{W}$ with the standard 
coefficients. 
From \cite[Proposition 3.4]{MumHir}, it follows that 
$(\overline{M_{W}},\bar{B})$ is the log canonical model 
and the three cusps $M_{W}^{\rm nn}$ and $M_{W}^{\rm seg}$,  
$M_{W}^{\rm nn}\cap M_{W}^{\rm seg}$ are the set of (all) 
log canonical centers. 

An important point to notice is that ${\rm Supp}(\bar{B})$ 
actually contain both of $M_{W}^{nn}$ and $M_{W}^{seg}$. Indeed, 
as we see below later, $M_{W}$ (with{\it out} the branch divisor) 
are log terminal around both log canonical centers. 

More direct way to see it is as follows. 
Recall from \cite[\S 7.1.5]{OO18} that 
the locus $S_{b}$ corresponding to 
$(b)$ in {\it loc.cit}, i.e., the 
surface in $M_{W}$ isomorphic to 
$\mathbb{A}^{1}\times \A^{1}$, parametrize 
Kummer surfaces for the product of elliptic curves 
$E_{1}\times E_{2}$ and the closure include 
both $M_{W}^{nn}$ and $M_{W}^{seg}$. 
As \cite[Proposition 7.8]{OO18} shows, 
for all such Kummer surfaces, the corresponding Weierstrass models  
contain 
four $D_{4}$-singularities which are ordinary cusps fiberwise, 
as a birational transform of $(E_{1}\times E_{2})/(\Z/2\Z)$. 
The Heegner divisor of $M_{W}$, which 
corresponds to their 
partial smoothings with a single $A_{1}$-singularity, 
contain the locus $S_{b}$ obviously. 

\subsubsection{Around $M_{W}^{\rm nn}$}

As the locus $M_{W}^{nn}(\setminus M_{W}^{\rm seg})$ locates inside 
the (strictly) stable locus inside the GIT quotient $\overline{M_{W}}$
 (cf., \cite[\S 7.1.1]{OO18}) 
it follows that the stabilizer of the $GL(2)$-action on 
$\mathbb{A}^{22}$ which represents a point inside 
$M_{W}^{nn}$ is finite. Furthermore, it is generically the 
Klein four group, i.e., $(\Z/2\Z)^{\oplus 2}$ and becomes 
larger only at finite points in 
$M_{W}^{nn}$ (e.g., when the corresponding degree $4$ polynomial 
$G_{4}$
is $s_{1}s_{2}(s_{1}-s_{2})(s_{1}+s_{2})$ (or $s^{3}-s$ in the way written in \cite{OO18}) 
so that the corresponding stabilizer group is 
$(\Z/2\Z)^{\oplus 3}$). 

Before our statements, we define the following 
singularity. 

\begin{Def}\label{K4.quot}
A canonical Gorenstein $3$-fold singularity whose germ is 
written as 
\begin{align}
\vec{0}\in [X^{2}=YZW]\subset \A^{4}
\end{align}
are denoted as $\mathcal{A}_{1}^{(3)}$ in this paper. 
Indeed, each component of the singular locus meeting at $\vec{0}$, 
\begin{itemize}
\item $X=Z=W=0, Y\neq 0$
\item $X=Y=W=0, Z\neq 0$
\item $X=Y=Z=0, W\neq 0$
\end{itemize}
are transversally $2$-dimensional $A_{1}$-singularity ($cA_{1}$), 
hence the name. 
It is also easy to see that this coincides with the quotient singularity 
by $(\Z/2\Z)^{\oplus 2}=K_{4}$ of $\A_{K}^{3}$ acting by the eigenvalues 
\begin{align*}
&(1,1,1,1)   &(\text{by the unit $e$ of }K_{4}),\\
&(1,-1,1,-1) &(\text{by an element $a$ of }K_{4}),\\
&(-1,-1,1,1) &(\text{by an element $b$ of }K_{4}),\\
&(-1,1,1,-1) &(\text{by the element $ab$ of }K_{4}). 
\end{align*}
\end{Def}

\begin{Thm}\label{nn.local}
At general points in $M_{W}^{nn}$, 
$M_{W}$ is formally (hence also analytically if $K=\C$) isomorphic to 
\begin{align}
(\mathcal{A}_{1}^{(3)}\times \mathcal{A}_{1}^{(3)}\times 
\mathcal{A}_{1}^{(3)}\times \mathcal{A}_{1}^{(3)})\times \A^{6},
\end{align}
hence canonical Gorenstein singular in particular. 
\end{Thm}
It is interesting as, with the branch divisor, 
it becomes one of strictly log canonical locus. 
\begin{proof}
We use the the Luna slice theorem \cite{Luna} (see also 
the exposition \cite[5.3]{Drezet}). Take a general point $p$ in $M_{W}^{\rm nn}$ 
and its lift $\tilde{p}$ to $\mathbb{A}_{g_{8},g_{12}}^{22}$ as 
$(P_{4}^{2},P_{4}^{3})$, where $P_{4}\in \mathcal{O}_{\mathbb{P}^{1}}(4)$ 
is of the form 
$(s_{1}^{2}-\epsilon^{2}s_{2}^{2})(s_{2}^{2}-\epsilon^{2}s_{1}^{2})$
so that its stabilizer is $K_{4}$ generated by 
\begin{align*}
(\text{switch}) \iota&\colon s_{1}\mapsto s_{2}, & s_{2}\mapsto s_{1} ,\\
(-1)_{s_{1}}&\colon s_{1}\mapsto -s_{1}, & s_{2}\mapsto s_{2}. 
\end{align*}

Now we construct slice at the above point in $\mathbb{A}_{g_{8},g_{12}}^{22}$ 
with respect to the natural ${\rm SL}(2)$-action as follows. 
Consider following 
regular parameter system (or holomorphic coordinates at neighborhood)
around $(P_{4}^{2},P_{4}^{3})\in \mathbb{A}^{22}_{g_{8},g_{12}}$: 
they are formed by $\coeff P_{4}$, the coefficients of the polynomial  $P_{4}$, 
which is introduced before, and those of 
\begin{align}
R^{\rm rfn}\in H^{0}(\mathbb{P}^{1},\mathcal{O}_{\mathbb{P}^{1}}(8)),\\ 
Q^{\rm rfn} \in H^{0}(\mathbb{P}^{1},\mathcal{O}_{\mathbb{P}^{1}}(4)),\\ 
R'^{\rm rfn}\in H^{0}(\mathbb{P}^{1},\mathcal{O}_{\mathbb{P}^{1}}(12)),
\end{align}
each of which are linear combinations of: 
\begin{itemize}

\item (for $R^{\rm rfn}$)
\begin{align*}
&s_{1}^{3}s_{2}^{5}\pm s_{1}^{5}s_{2}^{3}, \\ 
&s_{1}^{2}s_{2}^{6}\pm s_{1}^{6}s_{2}^{2} 
\end{align*}

\item (for $Q^{\rm rfn}$) 
\begin{align*}
s_{1}^{4}&\pm s_{2}^{4}, \\ 
s_{1}^{3}s_{2}&\pm s_{1}s_{2}^{3},\\ 
 &s_{1}^{2}s_{2}^{2}
\end{align*}

\item (for $R'^{\rm rfn}$)
\begin{align*}
s_{1}^{10}s_{2}^{2}&\pm s_{1}^{2}s_{2}^{10}, \\ 
s_{1}^{9}s_{2}^{3}&\pm s_{1}^{3}s_{2}^{9},\\ 
s_{1}^{8}s_{2}^{4}&\pm s_{1}^{4}s_{2}^{8},\\ 
s_{1}^{7}s_{2}^{5}&\pm s_{1}^{5}s_{2}^{7} 
\end{align*}
\end{itemize}
and we consider the points 
\begin{align}
(g_{8}=P_{4}^{2}+R^{\rm rfn}, g_{12}=P_{4}^{3}+(3P_{4})^{2}Q^{\rm rfn}+R'^{\rm rfn}), 
\end{align}
for those $R^{\rm rfn}, Q^{\rm rfn}, R'^{\rm rfn}$ 
which are generated by special ones above. 
Then this forms a ${\rm stab}(\tilde{p})$-invariant \'{e}tale slice. 
And the action of ${\rm stab}(\tilde{p})\simeq K_{4}$ whose generators we recall 
 as 
\begin{align}
(\text{switch}) \iota&\colon s_{1}\mapsto s_{2}, & s_{2}\mapsto s_{1} ,\\
(-1)_{s_{1}}&\colon s_{1}\mapsto -s_{1}, & s_{2}\mapsto s_{2},
\end{align}
act with eigenvalues $-1$ or $1$ on each basis vector above. 
Looking at the eigenvalues, the assertion readily follows. 
\end{proof}


\subsubsection{Around $M_{W}^{\rm seg}$}\label{MWseg}

Now, take a point $p\in M_{W}^{\rm seg}$ and its lift $\tilde{p}$ as 
$(c_{1}s_{1}^{4}s_{2}^{4},c_{2}s_{1}^{6}s_{2}^{6})$ for some $c_{1}, c_{2}\in K$, 
and consider the stabilizer group at the point with respect to the 
natural $GL(2)$-action, 
which we denote as ${\rm stab}(\tilde{p})$. 
It is simply isomorphic to 
$\mathbb{G}_{m}(K) \rtimes \mu_2(K)$ which acts as either 
\begin{align}
\{s_{1}\mapsto cs_{1}, s_{2}\mapsto c^{-1}s_{2}\mid c\neq 0\} \text{ or } \\ 
\{s_{1}\mapsto cs_{2}, s_{2}\mapsto c^{-1}s_{1}\mid c\neq 0\}.
\end{align} 

From the easy calculation of the tangent space to 
the orbit $GL(2)\tilde{p}$, 
we can take ${\rm stab}(\tilde{p})$-invariant 
\'etale slice at $\tilde{p}$ 
as 
$$\mathcal{S}(\tilde{p}):=
\tilde{p}+\{(\oplus_{0\le i\le 8, i\neq 3,4,5})k\cdot s_{1}^{i}s_{2}^{8-i}, 
\oplus_{0\le j\le 12}k\cdot s_{1}^{i}s_{2}^{12-i}\}\subset 
\mathbb{A}^{22}_{g_{8},g_{12}}. 
$$

Here we apply the Luna slice theorem \cite{Luna, Drezet} again to see the 
local structure around $M_{W}^{\rm seg}$. From above description 
of the slice $\mathcal{S}(\tilde{p})$, it is locally 

\begin{align}
(\mathcal{S}(\tilde{p})//(\mathbb{G}_{m}(K) \times \mu_2(K)) \equiv 
((\A_{K}^{18}// \mathbb{G}_{m}(K))/\mu_2(K)) \times K.
\end{align}
The weights for the $\mathbb{G}_{m}(K)$-action on $\A_{K}^{18}$ are twice the following 
\begin{align}
-4,-3,-2,2,3,4,
\end{align}
\noindent
as which correponds to the coefficients of 
$g_{8}$, further followed by 
\begin{align}
-6,-5,-4,-3,-2,-1,1,2,3,4,5,6
\end{align}
\noindent
as which correspond to the coefficients of 
$g_{12}$. 
Recall that in general, affine toric variety is characterized as 
GIT quotient of 
affine space by a linear action of some algebraic torus \cite[\S2]{Cox}. 
By applying it to our situation conversely, it follows that 
$\A_{K}^{18}//\mathbb{G}_{m}(K)$ is isomorphic to \footnote{this 
isomorphism is 
also easy to see directly, in this special case since the weights 
of the ${\rm stab}(p)(\simeq \mathbb{G}_{m}(K))$-action 
involve $1$ and the acting algebraic torus is one dimensional.}
the $17$-dimensional affine 
toric variety $U_{\sigma}$ corresponding to 
$\mathcal{S}_{\sigma}=\sigma^{\vee}\cap M$ 
defined as follows: 

\subsubsection*{Cone description}
if we consider $w\colon \R_{\ge 0}^{18}\to \R$ the inner product with the above 
vector 
$(-4,-3,-2,2,3,4,-6,-5,-4,-3,-2,-1,1,2,3,4,5,6)$, then 
for $\mathcal{S}_{\sigma}:=\Z^{18}\cap w^{-1}(0)$ and 
$\sigma:=\mathcal{S}_{\sigma}^{\vee}$ in the dual vector space 
$(\R^{18})^{\vee}$, 
above GIT quotient corresponds to this $\sigma\subset N\otimes \R$. 

\vspace{2mm}
It is easy to see this is nothing but the affine cone of 
self product of 
weighted projective space  
\begin{align}\label{exp.myst1}
\PP^{8}(1,2,2,3,3,4,4,5,6)\times \PP^{8}(1,2,2,3,3,4,4,5,6)
\end{align}
with respect to the $(\mathbb{Q}$-)line bundle $\mathcal{O}(1,1)$. 
Therefore, germ at any point in $M_{W}^{seg}$ in $\bar{M_{W}}$ 
is isomorphic to the product of smooth curve with 
the affine cone of ${\rm Sym}^{2}(\PP^{8}(1,2,2,3,3,4,4,5,6))$ 
with respect to the descend of $\mathcal{O}(1,1)$. 

Hence, if we blow up $M_{W}^{seg}$ with the descent of the vertex, 
we get 
\begin{align}\label{exp.myst2}
{\rm Sym}^{2}(\PP^{8}(1,2,2,3,3,4,4,5,6))
\end{align} as fibers over any point 
at $M_{W}^{seg}$. We suspect this corresponds to the variation of 
two rational elliptic surfaces.


\begin{Rem}
Looijenga 
\cite{Looi76} (cf., also 
Friedman-Morgan-Witten 
\cite[p.681-682]{FMW}) proves 
the following by use of the 
Weyl formula for affine root systems (Macdonald). 
We wonder if one can explain somewhat mysterious 
coincidence of the 
appeared exponents and those in \eqref{exp.myst1} 
and \eqref{exp.myst2}, 
in a more systematic manner. 

\begin{Thm}[{\cite{Looi76}, \cite{BS78}, cf., also Pinkham \cite{Pin77}, \cite{FMW}}]\label{roots.blowup}
For each elliptic curve $E$, and root lattice $Q$ and its dual root lattice 
$Q^{\vee{}}$, 
$(E\otimes Q^{\vee{}})/W(Q)$ is isomorphic to the weighted projective space of 
dimension ${\rm rk}(Q)$. The weights are e.g. 
\begin{align}
\mathbb{P}(\underbrace{1,1,1,1}_{4},\underbrace{2,2,2,2,2,2,2,2,2,2,2,2,2}_{l-3})
\end{align} for $D_{l}$ 
\begin{align}
\mathbb{P}(1,2,2,3,3,4,4,5,6)
\end{align} for $E_{8}$. 
Note that if $Q$ is of $A, D, E, F, G$ type, 
then $Q=Q^{\vee{}}$ by their self-duality. 
\end{Thm}
\end{Rem}


\section{Algebro-geometric compactification after \cite{ABE} \\ 
- elementary reconstruction - }\label{ABE.sec}

\subsection{Introduction to this section}

In this section, we reconstruct and analyze one of 
the algebro-geometric compactifications 
of $M_{W}$ recently studied in \cite[especially \S4C, \S7]{ABE}, 
denoted $F^{\rm rc}$ in {\it loc.cit}. 
There was also a preceding work \cite{Brun} before that, 
and there is also a closely related independent work \cite[especially \S 5 
and \S 9]{AB}. 
In this paper, we call the compactification $\overline{M_{W}}^{\rm ABE}.$ 
\cite{ABE} shows its normalization $\overline{M_{W}}^{\rm ABE,\nu}$ is a toroidal compactification, 
whose corresponding admissible rational polyhedral fan is what they 
call {\it rational curves fan} 
$\Sigma_{\rm rc}$ (\cite[\S4C]{ABE}), 
as introduced as ``$\mathcal{J}$" in \cite[Chapter 12]{Brun}, because the considered boundary on 
K3 surfaces are weighted 
sum of rational curves in the polarization, 
as in \cite{YZ,BL}. 

We briefly describe the points of our re-construction of $\overline{M_{W}}^{\rm ABE}$, especially the difference with 
\cite{ABE}. Our methods certainly overlap with the 
discussions in \cite{ABE} and even some exposition of this section 
\S \ref{ABE.sec} also parallel theirs, 
but the main point of our logic here is to 
replace some of essential parts of \cite{ABE} (especially the 
implicit/indirect 
stable reductions) 
by a simple elementary analysis of Weierstrass normal forms so that 
the construction extends even over $\mathbb{Z}[1/6]$. 
Also there is an independent nice work by \cite{AB}  
which constructs $\overline{M_{W}}^{\rm ABE}$ and described the 
boundary components in {\it loc.cit} section 9 (of version3), 
mainly from the viewpoints of the minimal model program again 
and twisted stable maps of \cite{AV}. 

In turn, our analysis mainly via Weierstrass equations helps the original differential geometric motivation shared with Y.Oshima after the paper \cite{HSZ} and 
fruitful discussions with S.Honda. Indeed, it is 
culminated in \S \ref{measure.decide} which decides 
very rich nontrivial moduli of 
all the {\it limit measures} of (further) 
Gromov-Hausdorff collapses from tropical K3 surfaces to 
an interval. For algebraic geometers, one can say that this gives a 
new invariant for type II degenerations of K3 surfaces, as a 
PL function of one real variable. 

As another virtue for algebaic perspective of the reconstruction, 
we also do not rely on 
the general theory of 
Koll\'ar-Shepherd-Barron-Alexeev moduli of semi-log-canonical models,  which 
in turn depends on the Minimal Model Program ($3$-dimensional 
relative semistable MMP in this case). 
Furthermore, from our construction, 
the presence of fibration structures on each degenerate surface come for free, 
which \cite[\S7C]{ABE} proved 
by some discussions on periods and deformation theory. 

Furthermore, our (re-)proof 
also do {\it not} logically use the tropical K3 surfaces or the 
key PL functions although we finally aim to clarify the 
meaning of those tropics appeared in \cite{ABE} and \cite{Osh}. 
We expect that this reconstruction also provides convenience for 
future study of limits of K3 metrics at {\it different rescale}. 

In this section, we first briefly review the irreducible components of 
stable degenerations introduced in \cite{ABE} (see also \cite[8.13]{AB}) and give 
alternative description to each. 


\subsection{Preparation}\label{ABE.pre.sec}

\subsubsection{Some notations}
\begin{itemize}
\item (recall) the base $\mathbb{P}^{1}$ of elliptic K3 surfaces 
in our concern, has homogeneous coordinates 
$s_{1}, s_{2}$ and $s:=s_{1}/s_{2}$. 
\item $g_{8}=\sum_{i}a_{i}s^{i}\in H^{0}(\PP_{s}^{1},\OO(8)=\OO(8[\infty]))$, 
\item $g_{12}=\sum_{i}b_{i}s^{i}\in H^{0}(\PP_{s}^{1},\OO(12)=\OO(12[\infty]))$, 
\item $\Delta_{24}=\sum_{i}d_{i}s^{i}\in H^{0}(\PP_{s}^{1},\OO(24)=\OO(24[\infty])).$ 
\item $g_{4}\in H^{0}(\PP_{s}^{1},\OO(4)=\OO(4[\infty]))$, 
\item $g_{6}=\in H^{0}(\PP_{s}^{1},\OO(6)=\OO(6[\infty]))$
\end{itemize}

\subsubsection{Degenerate surfaces over the compactified moduli by  \cite{ABE}}

We briefly recall that the degenerate surfaces over the 
boundary of $\overline{M_{W}}^{\rm ABE}.$ We explore and 
classify the prime divisors later in \S \ref{bd.div.sec}. 

First we focus on the type III degenerations parametrized on 
the normalization of $\overline{M_{W}}^{\rm ABE}$ 
i.e., the toroidal compactification 
$\overline{M_{W}}^{\rm toroidal,\Sigma_{\rm rc}}$ 
with respect to 
the rational curves cone 
$\Sigma_{\rm rc}$ 
(\cite[\S 12]{Brun}, \cite[\S 4C]{ABE}), which first 
parametrizes special Kulikov degenerations 
up to the flops of the ``Kulikov type'' either: 

\[
  \begin{cases}
   XI\cdots IX \\
   XI\cdots IY \\
   YI\cdots IY. 
  \end{cases}
\]
Each symbol refers to a irreducible components, but they are not all the 
components. 
We omitted the subindices (called ``charge'' as invariant of 
the integral affine singularities, in \cite{AET,ABE}), whose sum is $24$. 
When we pass to the ultimate 
KSBA degeneration, then many of the components 
are contracted so that we get a surface of the ``stable type'': 

\[
  \begin{cases}
   \D\A\cdots \A\D \\
   \D\A\cdots \A\E \\
   \E\A\cdots \A\E,
  \end{cases}
\]
respectively, as 
$X$ turns to $\E$ with subindex $3$ less, $Y_{2}Y_{d+2}$ turns to $\D$ with 
 (total) subindex $4$ less, 
 and $I$ turns to  $\A$ with subindex $1$ less during this 
contraction process. These $\A, \D, \E$ corresponds to the 
root lattices of the same symbols. 

From here, we recall some of the surface components including 
Type II case, and 
give some different elementary descriptions for our purpose of the 
reconstruction of $\overline{M_{W}}^{\rm ABE}$. 

\subsubsection{$\mathbb{A}$-type surface}\label{A.surf}

About the $\mathbb{A}$-type surface (\cite[\S7G]{ABE}), 
we have nothing new to add to \cite[\S7G]{ABE} so we simply recall it 
for readers' convenience. 
For the nodal rational curve $C$ i.e., 
the rational curve with only one singularity which is the node, 
consider $C\times \PP^{1}\to \PP^{1}$ 
with marked 
$k$ fibers over the points which are neither over 
$0$ nor $\infty$. 
The normalization is $\PP^{1}\times \PP^{1}\to \PP^{1}$. 



\subsubsection{$\mathbb{D}$-type surface}\label{D.surf}
For any square-free quadric polynomial $P_{2}$ of $s$, 
regarded as 
an element of 
$H^{0}(\mathbb{P}_{s}^{1},\mathcal{O}(2)=\mathcal{O}(2[\infty])$, 
the fibers of 
\begin{align}\label{D.surf.eqn}
X^{W}_{3P_{2}^{2},P_{2}^{3}}
:&=[y^{2}z=4x^{3}-3P_{2}^{2}xz^{2}
+P_{2}^{3}z^{3}]\to \mathbb{P}_{s}^{1} \\ 
 &=[y^{2}z=(2x-P_{2}z)^2(x+P_{2}z)]\\ 
&\subset  
\mathbb{P}_{\mathbb{P}_{s}}(\mathcal{O}_{\mathbb{P}
^{1}}(2)\oplus \mathcal{O}_{\mathbb{P}^{1}}(3)\oplus \mathcal{O}
_{\mathbb{P}^{1}}),
\end{align}
as fibration over $\mathbb{P}_{s}^{1}$, 
are generically (irreducible) nodal rational curves, 
with at most $2$ cuspidal rational curves over the roots of $P_{2}$. 

The normalization of this surface is the $\mathbb{P}^{1}$-
fiber bundle with fiber coordinates $[y:x-P_{2}z]$, 
which is $\PP_{\PP_{s}^{1}}(\OO_{\PP^1}\oplus \OO(1))$, 
the Hirzebruch surface $\F_{1}$. So, as  
$P_{2}$ is square-free, 
the surface coincides with 
the underlying fibered surface of $\D_{k}$-surface ($k$ 
only makes difference of the boundary divisors) which
\cite[\S 7G]{ABE} writes. 
The fiberwise ordinary cusps 
are simply pinch points as \cite{ABE}. 


\subsubsection{$\mathbb{E}$-type and $\tilde{\E}$-type surface}\label{E.surf}

For general $g_{4},g_{6}$, 
\begin{align}
X^{W}_{g_{4},g_{6}}:=[y^{2}z=4x^{3}-g_{4}
xz^{2}+g_{6}z^{3}]
\subset 
\mathbb{P}_{\mathbb{P}_{s}}(\mathcal{O}_{\mathbb{P}
^{1}}(2)\oplus \mathcal{O}_{\mathbb{P}^{1}}(3)\oplus \mathcal{O}
_{\mathbb{P}^{1}})
\end{align} 
becomes a rational elliptic surface 
with only ADE singularities. (cf., \cite{Mir,Kas}). 
We specify the $I_{9-k}$ Kodaira type fiber (\cite{Kod}) as the boundary, 
then we call this type of log surface $\E_{k}$ ($k=1$ has two types). 
If $k$ reaches $9$, we rather denote $\tilde{E}_{9}$ 
which is nothing but the rational elliptic surface minus a 
smooth elliptic curve fiber. 

Here, we allude to the fact that 
this $\E_{k} (k\le 8)$ surface (resp., $\tilde{E}_{9}$) 
is exactly the Landau-Ginzburg model for Del Pezzo surfaces 
(resp., rational elliptic surface) 
in the context of mirror symmetry as \cite{AKO} showed the homological 
mirror symmetry type statement.  Furthermore, the associated lattices coincides with those of Del Pezzo surfaces (\cite[Chapter IV, \S 25]{Manin}). 
See \cite{CJL} for related work.


\subsubsection{$\tilde{\mathbb{D}}$-type surface}\label{D.tilde}

We discuss $\tilde{\mathbb{D}}_{16}$-type surface similarly to 
above \S\ref{D.surf}. 
For a square-free quartic polynomial $G_{4}\in 
H^{0}(\mathbb{P}_{s}^{1},\mathcal{O}(4[\infty])$, we consider as in \cite[\S7]{OO18}
the explicit surface 
\begin{align}\label{D.tilde.eqn}
X^{W}_{3G_{4}^{2},G_{4}^{3}}
:&=[y^{2}z=4x^{3}-3G_{4}^{2}xz^{2}
+G_{4}^{3}z^{3}]\to \mathbb{P}_{s}^{1} \\ 
 &=[y^{2}z=(2x-G_{4}z)^2(x+G_{4}z)]\\ 
&\subset  
\mathbb{P}_{\mathbb{P}_{s}}(\mathcal{O}_{\mathbb{P}
^{1}}(2)\oplus \mathcal{O}_{\mathbb{P}^{1}}(3)\oplus \mathcal{O}
_{\mathbb{P}^{1}}). 
\end{align}
This is a generically nodal curve fibration, 
with exactly $4$ cuspidal rational curves degenerations over the roots of $G_{4}$ 
(see \cite[\S7.1.1 and \S7.1.3]{OO18} for details). 
The normalization of this surface is the $\mathbb{P}^{1}$-
fiber bundle with fiber coordinates $[y:x-G_{4}z]$, 
which is $\PP_{\PP_{s}^{1}}(\OO_{\PP^1}\oplus \OO_{\PP^1}(2))$, 
the Hirzebruch surface of degree $2$, i.e.,  $\F_{2}$. 

We remark here that the log KSBA surface parametrized along the 
same strata as \cite[\S 7F]{ABE} consists of $18$ components and 
the middle ruled components are all 
not open K-polystable in the sense of \cite{ops}, 
unless the $16$ $\mathbb{P}^{1}$s on the top components 
all the way flopped down to the bottom components 
so that all the middle components become trivial $\mathbb{P}^{1}$-bundle 
over the elliptic curve.


\subsubsection{Mutations of $Y$-surfaces}
Recall from \cite{ABE} 
that two type of parts of Kulikov degenerations $(Y_{2})Y_{2}(I_{a}\cdots)$ and 
$(Y_{2})Y'_{2}(I_{a}\cdots)$, modulo corner blowups, parametrized at the toroidal 
compactification $\overline{M_{W}}^{\rm toroidal,\Sigma_{\rm rc}}$ are 
{\it not} distinguished once we contract them to the KSBA models (they become 
$\D_{0}\A_{a-1}\cdots$ type) 
parametrized at $\overline{M_{W}}^{\rm ABE}$. This is the main reason of 
non-normality of $\overline{M_{W}}^{\rm ABE}$, as explained in 
\cite[\S7I]{ABE}. 

Here, we reinterpret this by {\it elementarily} (by explicit 
equations) 
construct a one parameter family 
of fibered surfaces 
at one parameter family level, hence total space $3$-dimensional
$\tilde{\pi}$ at \eqref{28.} soon 
by using only pure algebraic geometry of algebraic surfaces and simple birational 
geometry. At one end of the one parameter family, we have 
$Y_{2}(I_{a}\cdots)$ surface while the other end 
we see degeneration to $Y'_{2}(I_{a}\cdots)$. 
The generic fiber is $Y_{3}(I_{a}\cdots)$. 
This is the transition we should observe at the {\rm outer} (and left) 
part of the \cite[\S7]{ABE} type Kulikov degeneration. 

For that purpose, recall the Hirzebruch surface $\F_{1}$ and $\F_{0}=\PP^{1}
\times \PP^{1}$, $\PP^{1}$-bundles 
over the common base $\PP^{1}$ are elementary transforms of each other. 
Therefore, there is a common non-corner blow up which we write as 
$\varphi_{S}\colon S\to \mathbb{P}^{1}$ (this corresponds to $Y_{3}$ in \cite{ABE}) 
and we denote their centers in $\F_{i}$ 
are $p_{i} (i=0,1)$. We denote the projections 
as $\pi_{i}\colon \F_{i}\to \PP^{1}$ which satisfies 
$\pi_{0}\circ \varphi_{0}=\pi_{1}\circ \varphi_{1}$. 

In general, if we take a general conic in $\PP^{2}$ 
and its strict transform $D_{1}$ in $S, \F_{i} (i=0,1)$, 
then the projection to $\PP^{1}$ has 
two ramifying points as \cite[\S 7B]{ABE} write. 
It is easy to see that 
after the automorphism, we can and do assume that $p_{i}\in \F_{i}$ 
is one of two points $D_{1}\cap \pi_{S}^{-1}(\infty)$ for both $i$. 

Here we use the construction of \cite[\S 3.1]{Ohno}, 
which originally aimed to partially establish the 
CM degree minimization conjecture (cf., \cite{Ohno, Od.survey}) 
in the context of K-stability, 
by the author. 
One main point is we consider extra direction by introducing 
$\A^{1}_{t}$. We consider the blow up of 
$\PP^{1}\times \A^{1}_{t}$ at $(\infty,(t=)0)$ (resp., $(\infty,1)$), 
which we denote by 
\begin{align}
\beta_{i}\colon B_{i}\to \PP^{1}\times \A^{1}.
\end{align}
Then take the fibre product with 
\begin{align}
\Pi_{i}=(\pi_{i}\times {\it id})
\colon \F_{i}\times \A^{1}\twoheadrightarrow 
\PP^{1}\times \A^{1},
\end{align} for $i=0$ (resp., $i=1$) and 
further blow up the total space along a smooth closed curve 
$\overline{(\{\infty\}\times (\A^{1}\setminus \{i\}))}(\simeq \A_{K}^{1})$. 
Then we obtain\footnote{the author also used this construction 
in a joint work with R.Thomas on K-stability in 2013.}
\begin{align}
\tilde{\Pi_{i}}\colon \mathcal{F}_{i}\to {\rm Bl}_{(\infty,i)}
(\PP^{1}\times \A^{1}).
\end{align} We can glue these two for $i=0,1$, 
since the blow ups of $\F_{0}$ at $p_{0}$ and $\F_{1}$ at $p_{1}$ coincides, 
and obtain 
\begin{align}
&\hspace{1cm}\tilde{\Pi}&\colon &\mathcal{F} &\to Bl_{(\infty,0)\cup (\infty,1)}(\PP^{1}\times 
\A^{1}),\\
\label{28.}&\tilde{\pi}:= {\rm pr}\circ\tilde{\Pi}&\colon &\mathcal{F} &\to Bl_{(\infty,0)\cup (\infty,1)}(\PP^{1}\times 
\A^{1})\to \A^{1}. 
\end{align}
We denote the fiber over $t$ by $\bar{F}_{t}$. Then, $\mathcal{F}_{t}$ is 
\begin{align}
\label{4}
(\F_{0}\to \PP^{1}) & \cup (S\to \PP^{1}) &=Y_{2}'I_{a} & \text{ for } t=0 \\ 
\label{5}
S &\to \PP^{1} &=Y_{3}  &  \text{  for } t\neq 0,1 \\ 
\label{6}
(\F_{1}\to \PP^{1}) & \cup (S\to \PP^{1}) &=Y_{2}I_{a} & \text{ for } t=1. 
\end{align}
This interesting 
family $\{\mathcal{F}_{t}\}_{t}$ with two different degenerations at $t=0$ and 
$t=1$ 
exactly describes the switch between $Y_{2}Y_{2}$ and $Y_{2}Y_{2}'$ 
in the context of \cite{ABE}. 
Recall from \cite{ABE} (also see \cite{Osh}) 
that the corresponding PL functions to each of \eqref{4}, \eqref{5}, \eqref{6} 
starts with slope $8, 7, 8$ respectively. 

\subsubsection{Slight extension of ADE lattices}

In \cite{ABE}, over $K=\C$, they used the periods and corresponding 
Torelli theorems for components of the degeneration of 
elliptic K3 surfaces after \cite{GHK, Fri.recent}. 

The convention of denoting each components 
by $\A, \D, \E$ comes from it but 
for such description, they indirectly used the 
following slight extension of the usual ADE lattices; 
allows $D_{i}$ for $i=1,2,3$ and also $E_{i}$ for $i=1,2,3,4,5.$
We logically do not need it until \S\ref{Alg.limits.sec} but 
for the convenience of readers, we clarify here. 

The lattice $D_i$ for $i<4$ 
is constructed in the same way as those with $i\ge 4$. Simply, 
$$D_i:=\{ (x_1,\cdots,x_i)\in \Z^i \mid
\sum_j x_j \in 2\Z \}.$$
In our context, with respect to the fundamental domain, these D type lattices are naturally realized
in $\Lambda_{\rm seg}$ as
\begin{itemize}
\item 
$\langle \alpha_1-\alpha_3 \rangle$ for $i=1$,
\item 
$\langle \alpha_1,\alpha_3 \rangle$ for $i=2$,
\item 
$\langle \alpha_1,\alpha_3 ,\alpha_4 \rangle$ for $i=3$.
\end{itemize}

\vspace{3mm}
On the other hand, the following 
inductive construction of $E_i$ (from $i=1$) is 
essentially due to Manin \cite{Manin}. 

We construct a little extended lattice $E'_i$ 
for $i=1,2,\cdots$ with  $E_i\subset E'_{i}$ which has corank $1$ and orthogonal to $K_i$.
(Geometrically it is fairly simple i.e.,
$E'_{i}=H^2(S_{9-i},\Z)$ where $S_d$ stands for Del Pezzo surface of degree $d$ and
$c_1(S_d)^\perp = E_i$. )
Here is more elementary construction (through ``blow up”):
\begin{align*}
&E'_{i}&:= \Z l &(l^2=1), & -K_1=3l. \\ 
&E'_{i+1}&:=E'_{i}\oplus \Z e_i & (e_i^2=-1) & K_{i+1}=K_i+e_i.\\
\end{align*}

Allowing above type $D$ lattices and $E$ lattices with lower indices, we call these such A,D,E lattices and their direct sum as 
{\it slightly generalized root lattice}. 
See \cite[\S 1]{LO} for related discussions. 


\subsection{Re-construction of \cite{ABE}}\label{ABE.reconst.sec}

In our logic for the re-construction of the compactification of 
\cite{ABE}, first we readily construct the desired moduli stack $\overline{\mathcal{M}_{W}}^{\rm ABE}$ 
and then, we show the desired properties 
especially the properness 
as well as the presence of projective coarse moduli spaces 
$\overline{M_{W}}^{\rm ABE}$ ($F^{\rm rc}$ in \cite{ABE}) later. 

Our discussion uses the degenerations of the elliptic K3 surfaces parametrized 
by $\overline{M_{W}^{\rm ABE}}$ {\it simply as a set(!)} and denote them by 
$(X,R) \in \overline{M_{W}^{\rm ABE}}$. 
First we {\it fix} large enough positive integers $m$ and $d$ so that 
for any $(X,R=s+m\sum f_{i})\in \overline{M_{W}^{\rm ABE}}$, $R$ is 
ample and $dR$ is very ample without high cohomology. Obviously, 
 $\chi(X,\mathcal{O}_{X}(dR))$ does not depend on 
$(X,R)$s. Then we take the corresponding Hilbert scheme $H'$. 
Naturally, $G:={\rm SL}(H^{0}(X,dR))$ acts on $H$. 

We take a subset $H$ of $H'$ parametrizing the surfaces $X$ parametrized by $\overline{M_{W}}^{\rm ABE}$ embedded by $dR$. 
Since the subset is characterized as those $\mathcal{O}_{\mathbb{P}}(1)|_{X}
=\mathcal{O}_{X}(dR)$ (closed condition) as well as 
the reduced semi-log-canonical-Gorenstein 
properties of $X$ (open condition), $H$ is a locally closed subset of $H'$. 

Then we put reduced scheme structure on $H$ and 
set \begin{align}
\overline{\mathcal{M}_{W}}^{\rm ABE}:=[H/G], 
\end{align}
the quotient (a priori only Artin) stack. 
Now we prove this is actually a proper Deligne-Mumford stack (i.e., 
stable reduction type statements) case by case, so that we reprove the following 
in an elementary way. (Of course, we do not mean to be short arguments, by the word 
``elementary''.) 

\begin{Thm}[cf., \cite{ABE}]\label{ABE.moduli}
The moduli algebraic stacks (constructed above) 
$\mathcal{M}_{W}\subset \overline{\mathcal{M}_{W}}^{\rm ABE}$ of 
elliptic K3 surfaces and their degenerations over ${\rm Spec}(\Z[1/6])$, 
(the former is an open substack of the latter) 
both admit 
the coarse moduli varieties $M_{W}\subset \overline{M_{W}}^{\rm ABE}$ 
(the former is an open subvariety of the latter) such that 
$\overline{M_{W}}^{\rm ABE}$ is projective. 
\end{Thm}

\begin{proof}[Elementary direct reproof]

The existence of coarse moduli spaces as algebraic spaces follows from 
\cite{KeelMori}, 
since the inertia groups of the moduli stack are nothing but 
the automorphism of log canonical model $(X,\epsilon R)$ which is finite cf., 
\cite[Chapter 11]{Iit82}, \cite[Proposition 4.6]{Ambro}). 
The projectivity follows from the ampleness of the 
determinant of direct image sheaves of pluri-log-canonical bundles 
\cite{Kov},\cite{Fjn}. 

Therefore, to reprove Theorem~\ref{ABE.moduli}, 
it remains to show the following key claim from the valuative criterion of properness {\it relative to} ${\rm Spec}(\Z)[1/6]$ (e.g., \cite[\S7]{LM}). In particular, the uniqueness part 
shows that the reconstructed compactification in this section 
and \cite{ABE} are identical. 

\begin{Thm}[stable reduction cf., \cite{ABE}]\label{ABE.stable.reduction}

For any field $K$ of characteristic different from $2$ and $3$, and 
any $(X,R)\to \mathbb{P}^{1}_{s}$ parametrized in $\overline{\mathcal{M}_{W}}^{\rm ABE}
(K((t)))$, 
$(X,R)\to \mathbb{P}^{1}_{s}$
has a unique (explicit)
model $(\mathcal{X},\mathcal{R})\to \mathcal{B}$ over $K[[t]]$ 
in $\overline{\mathcal{M}_{W}}^{\rm ABE}
(K[[t]]).$  
\end{Thm}

We fix further notations before giving the details of the proof. 

\subsubsection*{Some further notations}
\begin{itemize}

\item $K$ denotes the field we take in Theorem~\ref{ABE.stable.reduction}, 
whose characteristic is coprime to $6$. 
Recall that we use $s$ for the 
corresponding coordinate, virtually valued in $K$.  

\item Since we only wish to prove properness of the above quotient algebraic stack, 
we can and do assume the field $K$ is actually algebraically closed, just for 
simpler exposition. 

\item We denote the obvious trivial model 
$\mathbb{P}^{1}_{s}\times {\rm Spec}(K[[t]])$
of 
$\mathbb{P}^{1}_{s}\times {\rm Spec}(K((t)))$ 
as $\mathcal{B}_{\rm triv}$. We make birational transforms of this 
$\B_{\rm triv}$ to other model $\B$. 

\item Discriminant locus of $[(X,R)\to \PP_{s}^{1}]
\in \mathcal{M}_{W}(K((t)))$ as $D\subset \B$. The fibers over 
its reduction 
$\overline{D}\cap (t=0)\subset \B$ are called 
{\it really singular} in \cite{ABE} which we continue to use. 
We call their underlying closed points in the base as {\it real discriminant (points)}. 
\end{itemize}

\begin{proof}[proof of Theorem~\ref{ABE.stable.reduction}]

The uniqueness part follows from the general uniqueness 
of relative log canonical model (i.e., which reduces to the 
independence of log canonical ring on any log smooth birational models 
cf., \cite{KM} for details) but also follows from the 
explicit analysis below. 

Hence, we focus on the explicit construction of 
the desired stable reduction to each punctured families lying 
on $\mathcal{M}_{W}$. 
By lifting to $\mathbb{A}^{22}$, 
reduce to the following four cases: Case \ref{Case1} 
to Case \ref{Case4}. 

\begin{Case}[Type III degenerations from $M_{W}$]\label{Case1}

This case amounts to show the following claim: 

\begin{Claim}[Maximally degenerating stable reduction]
\label{Max.degen.stable.reduction}

Given any $g_{8}(s)$ in $\Gamma(H^{0}(\mathbb{P}^{1}_{s},\mathcal{O}(8)))
\otimes K[[t]]$ 
(resp., $g_{12}(s)$ in $\Gamma(H^{0}(\mathbb{P}^{1}_{s},\mathcal{O}(12)))
\otimes K[[t]])$ such that 

\begin{align}
X^{W}_{g_{8},g_{12}}|_{t\neq 0}:&=[y^{2}z=4x^{3}-g_{8}(t)xz^{2}+g_{12}(t)z^{3}]\\ 
&\subset 
\mathbb{P}_{\mathbb{P}_{s}^{1}}(\mathcal{O}_{\mathbb{P}^{1}}(4)\oplus \mathcal{O}_{\mathbb{P}^{1}}(6)\oplus \mathcal{O}_{\mathbb{P}^{1}}),
\end{align}
as in \cite[\S7.1]{OO18} 
is an elliptic K3 surfaces parametrized in $M_{W}(K((t)))$ i.e., 
only with ADE singularities and 
$g_{8}|_{t=0}=3s^{4}, g_{12}|_{t=0}=s^{6}$ (i.e., 
converging to $M_{W}^{nn, seg}$ in the Satake-Baily-Borel compactification 
(cf., \cite[\S7]{OO18}), 
the corresponding $X\to \mathbb{P}^{1}_{s}$ (resp., 
$\mathbb{P}^{1}_{s}\times {\rm Spec}(K((t)))$) 
over $K((t))$ 
has another model 
$\mathcal{X}$ (resp., connected proper scheme $\mathcal{B}$ 
of relative dimension $1$) over $K[[t]]$ so that 
$\mathcal{X}|_{t=0}\to \mathcal{B}|_{t=0}$ is (the only possible) 
one of those parametrized in $\overline{M_{W}}^{\rm ABE}$. 
\end{Claim}


\begin{Stp}[End surfaces]\label{End.surfaces}
To prove the above Claim~\ref{Max.degen.stable.reduction}, 
first we take finitely ramified base change from $K[[t]]$ to 
$K[[t^{1/d}]]$ for some $d\in \mathbb{Z}_{>0}$, 
so that we can and do assume the roots of 
$g_{8}, g_{12}, \Delta_{24}:=g_{8}^{3}-27g_{12}^{2}$
 are Lawrent (not only Puiseux), i.e., there are 
$\xi_{i}\in K((t)) (i=1,\cdots,8)$, 
$\eta_{i}\in K((t)) (i=1,\cdots,12)$, 
$\chi_{i}\in K((t)) (i=1,\cdots,24)$ in the descending order of the valuations $v_{t}(-)$ along coordinates $s$ with respect to $t$ 
(or additive inverse of the 
valuation of $s^{-1}$). 
Here, $s' (:=\frac{s_{2}}{s_{1}})$ 
is regarded as a local uniformizer at 
$[s_{1}:s_{2}]=[1:0]$ (``$\infty$-point'') in 
the base $\mathbb{P}^{1}_{s}$. 

We first set 
\begin{align}\label{e.def}
e(0):=\min \{{\rm val}_{t}(\xi_{1}),\cdots,{\rm val}_{t}(\xi_{4}),
{\rm val}_{t}(\eta_{1}), \cdots, {\rm val}_{t}(\eta_{6})\},
\end{align}
\begin{align}\label{e'.def}
e(\infty):=\min \{{\rm val}_{t}\biggl(\dfrac{1}{\xi_{5}}\biggr),\cdots,{\rm val}_{t}\biggl(
\dfrac{1}{\xi_{8}}\biggr), 
{\rm val}_{t}\biggl(\dfrac{1}{\eta_{7}}\biggr), \cdots, 
{\rm val}_{t}\biggl(\dfrac{1}{\eta_{12}}\biggr)\}. 
\end{align}
and after an appropriate elementary transform of the trivially extended  
$\PP^{1}$-bundle over $\PP^{1}_{s}\times_{K}K[[t]]$ 
(we fix this ambiguity below soon), 
further blow it up to 
$\mathcal{B}_{1}\to \mathcal{B}_{\rm triv}$ by 
the coherent ideal sheaf 
\begin{align}
\langle s,t^{e(0)} \rangle \cdot \langle s', t^{e(\infty)}\rangle 
\cdot \mathcal{O}_{\mathcal{B}_{\rm triv}}. 
\end{align}
Then, the special fibre of $\mathcal{B}_{1}$ over $t=0$ 
is 
\begin{align}\label{base.step1}
\mathbb{P}^{1}_{\frac{s}{t^{e(0)}}}\cup 
\mathbb{P}^{1}_{s}
\cup \mathbb{P}^{1}_{\frac{s'}{t^{e(\infty)}}}
\end{align} where the 
two ends are exceptional curves. 

Accordingly, we can naturally degenerate the 
ambient space $\mathbb{P}_{\mathbb{P}_{s}^{1}}(\mathcal{O}_{\mathbb{P}^{1}}(4)\oplus \mathcal{O}_{\mathbb{P}^{1}}(6)\oplus \mathcal{O}_{\mathbb{P}^{1}})$ over $K((t))$ to 
over $K[[t]]$ so that the special fiber over $t=0$ is 
a connected 
union of the following three irreducible components: 
\begin{enumerate}
\item \label{1st.compo} 
$\mathbb{P}_{\mathbb{P}_{\frac{s}{t^{e(0)}}}^{1}}(\mathcal{O}_{\mathbb{P}^{1}}(2)\oplus \mathcal{O}_{\mathbb{P}^{1}}(3)\oplus \mathcal{O}_{\mathbb{P}^{1}})$ over $\mathbb{P}^{1}_{\frac{s}{t^{e(0)}}}$
\item \label{2nd.compo}
trivial $\mathbb{P}^{2}$-bundle 
over $\mathbb{P}^{1}_{s}$  (i.e., 
$\mathbb{P}^{2}\times \mathbb{P}^{1}_{s}$) 
\item \label{3rd.compo}
$\mathbb{P}_{\mathbb{P}_{s}^{1}}(\mathcal{O}_{\mathbb{P}^{1}}(2)\oplus \mathcal{O}_{\mathbb{P}^{1}}(3)\oplus \mathcal{O}_{\mathbb{P}^{1}})$ over $\mathbb{P}^{1}_{\frac{s'}{t^{e(\infty)}}}$. 
\end{enumerate}

Inside the first component \eqref{1st.compo}, 
the closure of $X$ (``limit component'') appears as 
\begin{align}
X^{W}_{g_{4}^{\nu},g_{6}^{\nu}}:=[y^{2}z=4x^{3}-g_{4}
^{\nu}|_{t=0}xz^{2}+g_{6}^{\nu}|_{t=0}z^{3}],
\end{align}
where 
$g_{4}^{\nu}=c_{4}\prod_{i=1}^{4}(s-\xi_{i}),$
$g_{6}^{\nu}=c_{6}\prod_{i=1}^{6}(s-\eta_{i}),$ 
with replaced roots $\xi$s and $\eta$s. 
Recall that construction of the model $\mathcal{B}_{1}$ above 
had an ambiguity modulo elementary transform with respect to 
$t=0$ but we fix it by assuming $(c_{4},c_{6})\in K^{2}\setminus 
\vec{0}$. 
From the construction, $g_{4}^{\nu}$ and $g_{6}^{\nu}$ are 
strictly degree $4$ and $6$ respectively with coefficients 
$3$ and $1$ respectively, 
$\Delta^{\nu}_{12}:=(g_{4}^{\nu})^{3}-27(g_{6}^{\nu})^{2}$ 
has degree at most $11$. This means 
the component 
$X^{W}_{g_{4}^{\nu},g_{6}^{\nu}}$ has singular fiber over 
$\infty$, which corresponds to the fact that the 
degeneration is of type III.

Also, from the definition of $e(0)$, 
not all of $\xi_{i}$s and $\eta_{i}$s vanish. 
Similarly, in the last component \eqref{3rd.compo}, 
the closure of $X$ (``limit component'') appears as 
\begin{align}
X^{W}_{h_{4}^{\nu},h_{6}^{\nu}}:=[y^{2}z=4x^{3}-h_{4}
^{\nu}|_{t=0}xz^{2}+h_{6}^{\nu}|_{t=0}z^{3}],
\end{align}
where 
$h_{4}^{\nu}=\prod_{i=5}^{8}(s-\xi_{i}),$
$h_{6}^{\nu}=\prod_{i=7}^{12}(s-\eta_{i}),$ 
again with newly replaced roots $\xi$s and $\eta$s. 
From the construction, due to \cite[Lemma1]{Kas}, if 
Weierstrass surfaces are generically smooth, 
they automatically only have ADE singularities (at non-zero finite base 
coordinates). 

When $K=\C$, 
in comparison with our asymptotic analysis of McLean's 
real Monge-Amp\'ere metrics 
in \cite[\S7.3.3]{OO18}, 
these ``end surfaces'' are where the term 
(denominator of the second term in \cite[Lemma 7.16]{OO18}) 
\begin{align}
\log(|g_{8}|^{3}+27|g_{12}|^{2})
\end{align} becomes dominant. 
On the other hand, the following 
next step is relevant to expand the divergence of the 
$\log(|\Delta_{24}|)$ term. 

\end{Stp}

\begin{Stp}[Separating ``middle'' $\chi_{i}$s]\label{Step.middle}
Next step we consider toric model $\B$ with respect to 
some combinatorial data coming from the Newton polygon, 
as the method used classically by \cite{Mum72.AV, 
AN, Don02} as follows. 
We consider the Newton polygon 
${\rm Newt}(\Delta_{24})$ of $\Delta_{24}$ i.e., 
the convex hull of 
\begin{align}\label{d.hull}
\{(i,v_{t}(d_{i}))\mid 0\le i\le 24\}+\R_{\ge 0}(0,1).
\end{align}
We regard it as a graph of PL convex function 
$\varphi_{\Delta}\colon [0,24]\to \R\cup \{\infty\}$. 
Then we modify this as follows (this process aims at 
including the previous step when we consider the toric models): 

Set 
\begin{align}
i_{e(0)}:=\max\{i\mid \varphi_{\Delta}(i)-\varphi_{\Delta}(i+1)\ge 
e(0)\},
\end{align}
\begin{align}
i_{e(\infty)}:=\min\{i\mid \varphi_{\Delta}(i+1)-\varphi_{\Delta}(i)\ge 
e(\infty)\},
\end{align}
where $e(0)$ and $e(\infty)$ as \eqref{e.def} and \eqref{e'.def}. 
We modify $\varphi_{\Delta}$ to 
$\overline{\varphi}_{\Delta}\colon [0, 24]\to \R\cup \{\infty\}$ defined as follows: 
\begin{align}\label{PL.modif}
\bar{\varphi}_{\Delta}(i):=
\begin{cases}
\varphi_{\Delta}(i_{e(0)})-e(0)(i_{e(0)}-i) & (\text{if }0\le i\le i_{e(0)})  \\ 
\varphi_{\Delta}(i) & (\text{if }i_{e(0)}\le i\le i_{e(\infty)})   \\
\varphi_{\Delta}(i_{e(\infty)})+e(\infty)(i-i_{e(\infty)}) & (\text{if }i_{e(\infty)}\le i\le 24). 
\end{cases}
\end{align}

Then, consider the toric model (test configuration of $\PP^{1}$) 
$\B$ over $\A^{1}$ (hence also over $K[[t]]$), 
corresponding to $\bar{\varphi}_{\Delta}$, i.e., 
for 
\begin{align}\label{P.Delta}
P_{\Delta,c}:=\{(x,y)\in \R^{2}\mid 0\le x\le 24, 
-c\le y\le -\bar{\varphi}_{\Delta}(x)\}
\end{align}
for some $c\ge 0$ the moment polytope of (the natural 
compactification of) $\B$ becomes 
$P_{\Delta,c}$. 

In particular, 
the normal fan of the graph of $\bar{\varphi}_{\Delta}$ 
gives $\B$ by usual toric construction. We fix and take the natural 
$c$ such that the 
obtained $\B$ has the same end components as 
$\mathcal{B}_{1}$ in the previous Step \ref{End.surfaces}, 
i.e., the end components of $\B|_{t=0}$ 
are the bases 
$\mathbb{P}^{1}_{\frac{s}{t^{e(0)}}}$ and 
$\mathbb{P}^{1}_{\frac{s'}{t^{e(\infty)}}}$ 
of the ends at \eqref{base.step1}. 
Indeed, it is possible by our modification \eqref{PL.modif} of the 
PL function. 

Furthermore, as desired, 
every other components of $\B|_{t=0}$ has 
at least one point of $D_{0}(=\bar{D}\cap (t=0))$. 
Here, recall that $D$ denotes the discriminant locus defined after Theorem \ref{ABE.stable.reduction} whose closure is denoted as $\bar{D}$. 
This ensures the ampleness of the boundary $R$ in the corresponding irreducible components of the Weierstrass (reducible) fibred surface. 
\end{Stp}

\begin{Stp}[About end surfaces again]\label{Step.1.3}
If the end surface $X^{W}_{g_{4}^{\nu},g_{6}^{\nu}}\to \PP^{1}_{s}$ 
is generically smooth, it is nothing but a rational elliptic surface 
i.e., type $\E_{k}$ in \cite{ABE}. In that case, because of the construction, $\deg \Delta_{12}^{\nu}=12$ ($\Delta_{12}^{\nu}$ does not vanish at $\infty$) so that 
the fiber over $\frac{s}{t^{e(0)}}=\infty$ 
can not be singular. 

On the other hand, 
if the end surface $X^{W}_{g_{4}^{\nu},g_{6}^{\nu}}$ has 
singular general fibers, it means that there is 
$P_{2}\in H^{0}(\OO(2[\infty]))$ 
such that 
\begin{align}
g_{4}^{\nu}=3P_{2}^{2}, g_{6}^{\nu}=P_{2}^{3}.
\end{align}
${\rm deg}(P_{2})$ can not be less than $2$ from the construction. 
If this $P_{2}$ is square-free, then from our discussion in 
\S \ref{D.surf}, 
we get the surface $\D$ type and end the step here. 
If $P_{2}$ is {\it not} square-free, 
we continue to next step. 
\end{Stp}

\begin{Stp}[Modifying almost $\D$ type end]\label{Real.end}
Depending on formulation, this process may be included in Step\ref{End.surfaces} 
but nevertheless we separated it to make the steps clearer. 
From here, we treat the ``left end'' surfaces in the original sense of 
Step\ref{End.surfaces} i.e., those maps to $s=0$ i.e., 
defined 
by $g_{4}^{\nu}$ and $g_{6}^{\nu}$. 
(For the right end surface which maps to $s=\infty$, 
the completely similar arguments work by symmetry so we avoid repitetion of the 
details of the arguments.) 

We continue from the previous step, so 
suppose $P_{2}$ is {\it not} square free. 
Nevertheless, since our generic fiber at $t\neq 0$, 
$X_{g_{8},g_{12}}^{W}$ was originally at worst ADE, 
among those (a priori at total $10$) roots of $g_{4}^{\nu}$ or $g_{6}^{\nu}$ i.e., 
$\xi_{i} (1\le i\le 4), \eta_{j} (1\le j\le 6)$, 
at least two of them do not coincide 
as elements of $K[[t]]$ (before substitution $t=0$). 
Suppose that they are 
$\{p,q\}\subset \{\xi_{i} (1\le i\le 4), \eta_{j} (1\le j\le 6)\}$ 
with respect to the new coordinates after Step~\ref{End.surfaces}. 
Write the local uniformizer at $p(0)=q(0)$ for the component 
$\mathbb{P}^{1}_{\frac{s}{t^{e(0)}}}$, 
as $s_{p,q}$. 

We make Puiseux expansions of $p, q$ and set 
$e_{p,q}:=v_{t}(p-q)$, where $v_{t}$ denotes the 
$t$-adic (additive) valuation. 
Then do blow up of $\B$ (which was the outcome of processes until the previous 
step) along $\langle s_{p,q}, t^{e_{p,q}}\rangle 
\mathcal{O}_{\B}$ whose 
cosupport is in $\mathbb{P}^{1}_{\frac{s}{t^{e(0)}}} \times \{t=0\}$, 
and blow down the surface without $\bar{D}\cap (t=0)$ if necessary, 
we obtain the situation with squarefree $P_{2}$. 
Note that by this last step, the resulting model 
$\B$ may {\it not} be toric, while toroidal, with respect to the original 
coordinates (since $p(0)=q(0)$ may not be zero). 

\end{Stp}

\end{Case}
\vspace{4mm}

\begin{Case}[Type II degenerations]
These cases are essentially done in \cite[\S 3]{CM05} via deformation theory and 
more Hodge-theoritic viewpoint, while 
the degenerations are slightly modified in \cite{ABE} (see also 
\cite{Fri84, Kondo} including non-elliptic case). 

Here we again recover them by our 
elementary method using the Weierstrass form as below. 

\begin{Subcase}(to $\widetilde{\D_{16}}$) 
This case essentially follows from the 
GIT picture in  \cite[\S7]{OO18} by applying the GIT stable reduction. Recall that 
the Satake-Baily-Borel compactifiation $\overline{M_{W}}^{\rm SBB}$ 
coincides with the GIT compactification with respect to the 
Weierstrass expression \cite[\S 7.2.1]{OO18}. 
As \cite[\S 7.1]{OO18} shows, the locus $M_{W}^{\rm nn}$ 
is in the strictly stable locus, which parametrizes the 
semi-log-canonical surface of the form \eqref{D.tilde.eqn}, 
which is nothing but $\tilde{\D}_{16}$-type in \cite{ABE}. 

If we have $(g_{8},g_{12})\in H^{0}(\OO(8))\times H^{0}(\OO(12))$ 
over the base $K[[t]]$, with reduction sits in the stable locus 
mapping down to $M_{W}^{\rm nn}$, then the GIT stable reduction 
proves that after finite base change if necessary, if we apply 
an element of $SL(2)$ in the coefficient $K((t))$, we get reduction 
with special fiber of the surface of type \eqref{D.tilde.eqn}. This completes the 
required process. 

\begin{Rem}
By comparing with toroidal compactification, recall that 
Type II locus does not depend on the choice of admissible rational 
polyhedral decompositions (cf., e.g., \cite{Fri}). 
Furthermore,  the preimage of $M_{W}^{\rm nn}$ in it 
which we write as $M_{W}^{\rm nn, tor}$ 
is a ${\rm Aut}(D_{16})$-quotient of the $16$-th self fiberproduct of 
the (coarse moduli of) universal elliptic curve over $M_{W}^{\rm nn}\simeq 
\A_{j}^{1}$ ($j$ stands for the $j$-invariant of $E$). There is a very 
clear geometric meaning to this phenomenon -  
by \cite[\S3]{CM05} and \cite[7.20, 7.22, 7.44]{ABE}, 
the $16$ real discriminants are arbitrary (for each fixed $E$), 
which give the difference of 
this $M_{W}^{\rm nn}$ and $M_{W}^{\rm nn, tor}$. 

Note that the parametrized degeneration is slightly different between 
that in \cite[\S3]{CM05} and \cite{ABE} (i.e., the former has two components 
one of which is those parametrized in \cite{ABE} - $\D_{16}$-surface), 
but this is unsubstantial difference. Indeed, the relation is by 
a simple birational transform (at the total space level) as explained in 
\cite[\S 7.1.3]{OO18}. 
\end{Rem}

\end{Subcase}

\begin{Subcase}(to $\widetilde{\E_{8}}\widetilde{\E_{8}}$) 
We treat the case of degenerating from $M_{W}$ to $M_{W}^{\rm seg}
\subset M_{W}^{\rm SBB}$, which we recall to be the 
$\tilde{E}_{8}^{\oplus 2}$-type $1$-cusp (see also its GIT interpretation in \cite[\S7]{OO18}). 

Take $(X,R)\twoheadrightarrow B$ in $\mathcal{M}_{W}(K((t)))$ which degenerates to 
$M_{W}^{\rm seg}$ at the closed point. 
From \cite[\S7]{OO18}, it follows that we can lift this data to 
$(g_{8}, g_{12})\in H^{0}(\OO(8))\times H^{0}(\OO(12))$ with coefficients in 
$K[[t]]$ so that its reduction is $(cs^{4},s^{6})$ for $c\neq 3$. 

Then we can exploit the same procedure as Case\ref{Case1} Step~\ref{End.surfaces}, 
to replace the reduction as the reducible fibered surface 
\begin{align}
(X_{1}\cup X_{2})\to \PP^{1}\cup \PP^{1}
\end{align} where 
$X_{1}$ (resp., $X_{2}$) is a of the $\PP^{2}$-bundle 
$\PP_{\PP^{1}}(\OO_{\PP^1}\oplus \OO_{\PP^1}(2)\oplus \OO_{\PP^1}(3))$
over the first $\PP^{1}$ (resp., the second $\PP^{1}$), 
defined by 
\begin{align}
[y^{2}z=4x^{3}-g_{4}^{\nu}xz^{2}
+g_{6}^{\nu}z^{3}], \\ 
[y^{2}z=4x^{3}-h_{4}^{\nu}xz^{2}
+h_{6}^{\nu}z^{3}],
\end{align}
respectively. 
Then, from our assumption that $c\neq 3$, it follows that the 
double locus $X_{1}\cap X_{2}$ is smooth elliptic curve fiber, 
hence this is of $\tilde{E}_{8}\tilde{E}_{8}$-type surface as desired. 
We have $12$ real discriminant points in each base. 
\end{Subcase}

\end{Case}

\vspace{4mm}

\begin{Case}[Further degenerations from Type III degenerations]

Below, 
we study the occuring degeneration componentwise. We proceed as follows. In the notations below, we promise that 
\begin{enumerate}
\item $\sum l_{i}=l$, 
\item all the subindices are nonnegative, 
\item We call the images of really singular fibers (cf., notations below 
Theorem~\ref{ABE.stable.reduction}) 
on any of 
{\it possibly singular } 
$[(X,R)\to B(\simeq \PP^{1}\cup\cdots \cup \PP^{1})] 
\in \overline{M_{W}}^{\rm ABE}(K) \text{ or } \overline{M_{W}}^{\rm ABE}
(K((t)))$ as 
$\chi_{1},\cdots,\chi_{24}$ (which extends the original meaning 
in the realm of $M_{W}$) and continue to call them real discriminant points. 
\item Further, before each disucssion below, we lift this data $[(X,R)\to B\simeq \PP^{1}\cup\cdots \cup \PP^{1}]$ 
by fixing gauge i.e., the isomorphism of every rational component with $\PP^{1}$ 
so that their nodal points have coordinate $0$ or $\infty$. 
\end{enumerate}

\begin{Subcase}($\A_{l-1}$ to 
$\A_{l_{1}-1}\A_{l_{2}-1}\cdots \A_{l_{m}-1}$) 
We now concentrate on the base of component of $A$-type in the 
degenerated 
\begin{align}
[(X,R)\to \PP^{1}  \cup\cdots \cup \PP^{1}]\in \overline{M_{W}}^{\rm ABE}
(K((t)))
\end{align} which we denote as 
$X_{A}\to \PP^{1}$ here, with coordinate $s_{A}$. The real discriminant points 
$\chi_{a+1},\cdots,\chi_{a+l}$ can be seen as formal Puiseux series 
i.e., elements of 
$\overline{K((t))}$. Note that any of $\chi_{a+i}$ is {\it not} $0$ nor 
$\infty$ (as element of $\PP^{1}(K((t)))$. 
Hence, after finite base change, we can 
suppose they all lie in $K((t))$ and we write 
$\Delta_{A}(s_{A}):=\prod_{1\le i\le l} (s_{A}-\chi_{a+i})$. 

Similarly to Step \ref{Step.middle} of Case \ref{Case1}, 
we take Newton polygon ${\rm Newt}(\Delta_{A})$, its 
supporting function $\varphi_{A}$ and the toric 
degeneration model $\B_{A}$ over $\A^{1}$ (hence also over $K[[t]]$) 
whose corresponding fan is 
the normal fan of the graph of $\varphi_{A}$. Or in other words, 
the natural compacfication has moment polytope 
\begin{align}
\{(x,y)\in \R^{2}\mid 0\le x\le l, 
-c\le y\le - \varphi_{A}(x)\}
\end{align} 
for a constant $c\gg 0$. 
This is one component of our desired $\B$ i.e., the closure of $\PP^{1}_{s_{A}}$. 
Then, accordingly, we degenerate the ambient space $\mathbb{P}^{2}\times 
\PP_{s_{A}}^{1}\simeq \PP_{\PP_{s_{A}}^{1}}(\OO^{\oplus 3})$ 
to still trivial $\PP^{2}_{s_{A}}$-bundle over $\B$ so that 
we obtain the (semi-log-canonical) 
union of $\A$-type log surfaces as the closure of $X_{A}$ 
inside the ambient model $\B\times \PP^{2}$. 

\end{Subcase}

\begin{Subcase}($\D_{k+l}$ to $\D_{k}\A_{l_{1}-1}\cdots\A_{l_{m}-1}$) 
Next we consider the base of component of $D$-type in the 
degenerated 
\begin{align}
[(X,R)\to \PP^{1}  \cup\cdots \cup \PP^{1}]\in \overline{M_{W}}^{\rm ABE}
(K((t)))
\end{align} which we denote as 
$X_{D}\to \PP^{1}$ here, with coordinate $s_{D}$. 
We can and do suppose that the only double curve in $X_{D}$ 
which is the intersection of next surface component, 
has coordinates $s_{D}=\infty$. 

Recall from \S\ref{D.surf} that we have explicit Weierstrass type equation 
for the $\D$-type surface, 
\eqref{D.surf.eqn} in terms of a quadratic polynomial $P_{2}(s_{D})$ 
whose coefficients live in $K((t))$. By quadratic base change if necessary, 
we can 
further suppose its two roots are also both in $K((t))$. 
Then by multiplying appropriate powers $t^{2c}$, $t^{3c}$ of $t$ to $g_{4}^{\nu}$ 
and $g_{6}^{\nu}$ which does not change the isomorphism class of 
original $X_{D}\to \mathbb{P}^{1}$ (over $t\neq 0$), we can and do assume that coefficients of both lie in 
$K[[t]]$ and do not vanish at $t=0$ generically (with respect to $s_{D}$). 

If some of real discriminants $\chi_{i}$ in the base of $X_{D}$ 
(including two roots of $P_{2}$)  
converges to $\infty$, whose fiber is in the double locus 
of the surface, then we do weighted blow up of the model 
finite times 
so that all $\chi_{i}$ in the base of $X_{D}$ never diverge to $
\infty$ 
when $t=0$. Furthermore, in a similar manner, if two distinct roots of 
$P_{2}$ converges to the same point for $t\to 0$, 
then we do further weighted blow up at the point 
so that the two roots converge to different points. 
After these composition of weighted blow ups of the base surface, 
we contract all irreducible components of $t=0$ 
which do not contain  any real discriminant. 

Then, we degenerate the bundle $\OO_{\PP^1}\oplus \OO_{\PP^1}(2)\oplus \OO_{\PP^1}(3)$ on 
$\B_{t\neq 0}$ to the whole model obtained above, 
so that it restricts to 
 $\OO_{\PP^1}\oplus \OO_{\PP^1}(2)\oplus \OO_{\PP^1}(3)$ on the component where the roots of $P_{2}$ 
 converge, to $\OO_{\PP^{1}}^{3}$ otherwise. We consider its projectivization 
 as the ambient model and take closure of $X$ to get desired model of 
 type $\D\A\cdots\A$. 

\end{Subcase}

\begin{Subcase}($\E_{k+l}$ to $\E_{k}\A_{l_{1}-1}\cdots\A_{l_{m}-1}$)
Next we consider the base of component of $\E$-type in the 
degenerated 
\begin{align}
[(X,R)\to \PP^{1}  \cup\cdots \cup \PP^{1}]\in \overline{M_{W}}^{\rm ABE}
(K((t)))
\end{align} which we denote as 
$\pi_{E}\colon X_{E}\to \PP^{1}_{s_{E}}$ here, with coordinate $s_{E}$. 
We can and do suppose that the only double curve in $X_{E}$ 
which is the intersection of next surface component, 
has coordinates $s_{D}=\infty$. 
We consider stable reduction of generic fiber thus over $K(s_{E})$, 
which is from elliptic curve to either elliptic curve or 
(irreducible) nodal rational curve over whole $K(s_{E})[[t]])$. 
Correspondingly, we realize this model by 
multiplying $t^{2c}$ (resp., $t^{3c}$) 
$$g_{4}  \in 
H^{0}(\PP^{1}_{s_{E}},\OO(4)) (\text{resp.}, g_{6} \in H^{0}(\PP^{1}_{s_{E}},\OO(6)))$$ 
 with appropriate $c$ (we fix this 
normalization from now on), 
so that $g_{4}, g_{6}$ both become non-zero at $t=0$. 

In this subcase, we focus when the generic fiber at $t=0$ is 
smooth i.e., elliptic curve, which we suppose from now on, 
and leave the nodal reduction case to next subcase. 

Suppose the real discriminant points $\chi_{e+1},\cdots,\chi_{e+k+l+3}$ 
below $X_{E}$ 
also all sit in $K((t))$ after finite base change if necessary. 
Then in a similar manner as before, 
with respect to the variable $s'_{E}:=s_{E}^{-1}$, we set 
\begin{align}
P_{E}(s'_{E}):=\prod_{i=1}^{3+k+l} (s'_{E}-\chi_{e+i}^{-1}),
\end{align}
consider its Newton polygon ${\rm Newt}(P_{E})$, 
then corresponding toric blow up model 
$\B_{E}\to \PP_{s_{E}}^{1}\times {\rm Spec}(K[[t]])$  
with cosupport at $t=0, s_{E}=\infty$. 
Then, generalizing the stable reduction over $K(s_{E})[[t]]$, 
we extend ambient space $\PP_{\PP_{s_{E}}^{1}}(\OO_{\PP^1}\oplus 
\OO_{\PP^1}(2)\oplus \OO_{\PP^1}(3))$ of $X_{E}$ to that of $\B$ so that 
its restriction to $\PP_{s_{E}}^{1}$ is $\PP(\OO_{\PP^1}\oplus 
\OO_{\PP^1}(2)\oplus \OO_{\PP^1}(3))$ which includes the $t$-direction 
stable reduction of the generic fiber of $X_{E}$, and trivial 
$\PP^{1}$-bundle over the rest of components of $\B|_{t=0}$. 
Then it is easy to see the closure inside the ambient model over $K[[t]]$ 
gives reduction to the surface of type $\E\A\cdots \A$ in \cite{ABE}. 
\end{Subcase}

\begin{Subcase}($\E_{k+l}$ to $\D_{k-1}\A_{l_{1}-1}\cdots\A_{l_{m}-1}$)
Similarly to the previous subcase, we next treat the case when 
$t$-direction stable reduction of the generic fiber of $X_{E}$ becomes 
nodal (i.e., $j=\infty$). This assumption means 
\begin{align}
\Delta_{12}=(g_{4}^{\nu})^{3}-27(g_{6}^{\nu})^{2}=0
\end{align} 
hence we can write $g_{4}^{\nu}=3P_{2}^{2}, g_{6}^{\nu}=P_{2}^{3}$. 
Since we normalized our $g_{i}^{\nu}$ to give the $t$-direction 
stable reduction of the generic fiber, $P_{2}|_{t=0}\neq 0$ as a polynomial. 

If the roots of $P_{2}|_{t=0}$ remain finite and distinct, 
then we only need to do toric modifications of the base model 
$\mathbb{P}^{1}_{s_{E}}\times {\rm Spec}K[[t]]$ at cosupport 
$\infty \times 0(\text{closed point})$. 
As it is completely similar to the just previous subcase, using 
Newton polygon of the polynomial of $s'_{E}$ with roots $\chi_{i}^{-1}$ 
converging to $0$, we omit details. 

If at least one the roots of $P_{2}|_{t=0}$ diverge, 
then we do toric blow up at 
$\infty \times 0\in \mathbb{P}^{1}_{s_{E}}\times {\rm Spec}K[[t]]$ 
so that $\B|_{t=0}$ becomes union of $\mathbb{P}^{1}_{s}$ with one or two 
exceptional divisors to each of which the diverging real discriminant 
converge. 
Also, if the roots of $P_{2}$ converge to same points $q$ in 
$\mathbb{P}^{1}_{s_{E}}$, we do weighted blow up 
of the base model surface 
at the point $q$ so that the roots converge to different points 
in the same component which we (still) denote as $\mathbb{P}^{1}_{s}$. 
After that, we contract all irreducible components (curves) 
of $t=0$ which do not contain any real discriminant. 
Then  again similarly, 
we take ambient space whose restriction to 
$\mathbb{P}^{1}_{s}\times \{t=0\}$ (resp., other components) 
is $\PP_{\PP^{1}}
(\OO_{\PP^1}\oplus \OO_{\PP^1}(2)\oplus \OO_{\PP^1}(3))$ (resp., trivial $\mathbb{P}^{2}$-bundle). 

After all these procedures, 
we obtain the model of reduction type $\D\A\cdots\A$. 

\end{Subcase}
\end{Case}

\vspace{4mm}

\begin{Case}[From Type II to Type III]\label{Case4}

Now we deal with the case when the corresponding morphism from 
${\rm Spec}K[[t]]\to \overline{M_{W}}^{\rm SBB}$, where the 
target space refers to the Satake-Baily-Borel compactification, 
maps generic point 
inside $1$-cusp ($M_{W}^{\rm seg}$ and $M_{W}^{\rm nn}$ in the \cite[\S7]{OO18} 
notation), and maps the closed point to $0$-cusp $M_{W}^{\rm nn,seg}$. 
We assume this below and call it $(*_{II,III})$. 

\begin{Subcase}($\widetilde{\E_{8}}\widetilde{\E_{8}}$ to $\E_{9-l}\A_{l_{1}-1}\cdots\A_{l_{m}-1}$)
First, we treat the case when the generic point of ${\rm Spec}K[[t]]$ 
maps to $M_{W}^{\rm seg}$. (Other case when the generic point of 
${\rm Spec}K[[t]]$ maps to 
$M_{W}^{\rm nn}$, is treated in the {\bf Subcase} after next.) 
We write the component of $\E_{9}$-surface (\cite{ABE}) i.e., 
rational elliptic surface with double locus a single smooth fiber, 
as $X_{E}\to B_{E}\simeq \PP^{1}$ as local notation. We suppose the double locus 
fibers over $\infty$. 

In case the reduction $t=0$ gives divergence of some real discriminants in 
the base $B_{E}$ 
to $\infty$, then we again do the toric blow ups of the model completely 
similarly as in previous steps via Newton polygon technique, 
so that the real discriminant points 
only converge finite in the strict transform of $B_{E}$ and 
smooth points in $\B|_{t=0}$ in general. Then again in the similar manner, 
we obtain model of polarization whose restriction to $B_{E}$ 
is $\OO_{\PP^1}\oplus \OO_{\PP^1}(2)\oplus \OO_{\PP^1}(3)$ while trivial $\OO^{\oplus 3}$ 
otherwise, projectify it and take closure of $X_{E}$ inside. 

If such model is generically smooth over the strict transform of $B_{E}$ (
otherwise, proceed to next subcase). 
Then by the assumption $(*_{II,III})$ it follows that the 
fiber over $\infty$ becomes nodal at $t=0$ (otherwise, it remains to be 
in Type II locus  i.e., $1$-cusps of $\overline{M_{W}}^{\rm SBB}$. 
Hence the reduction for $t=0$ is the desired fibred surface of type 
$\E\A\cdots\E$. 

\end{Subcase}

\begin{Subcase}($\widetilde{\E_{8}}\widetilde{\E_{8}}$ to $\D_{8-l}\A_{l_{1}-1}\cdots\A_{l_{m}-1}$)
If the obtained model of $(X_{E},R)\to B_{E}$ in the last step 
is {\it not} generically smooth 
over the strict transform of $B_{E}$, then the corresponding 
elements of $H^{0}(\OO_{\PP^{1}}(4))$ (resp., $H^{0}(\OO_{\PP^{1}}(6))$) 
which we still prefer to write $g_{4}^{\nu}$, $g_{6}^{\nu}$ 
are of the form $(3P_{2}^{2},P_{2}^{3})$ with some 
$P_{2}\in H^{0}(\OO_{\PP^{1}}(2))$. If $P_{2}$ vanishes at $\infty$, 
i.e., degree at most $1$ as a polynomial, then 
it means that one of the root of $P_{2}$ which is also a real discriminant point, 
diverges (or converges) to $\infty$. 
We 
do toric blow up of the model of $B_{E}$ at this stage by the Newton polygon of 
the polynomial whose roots are diverging real discrminants, 
as in the previous steps. The process avoids 
the divergence of real discriminants $\infty$ 
while 
procuding further rational components in the reduction of base $\B|_{t=0}$. 
If $P_{2}$ is {\it not} squarefree, 
we do the same process as Case\ref{Case1} Step\ref{Real.end}. 
Then we contract all irreducible components of $t=0$ 
which do not contain any real discriminants. 

Then finally, similarly, we create the model of $\OO_{\PP^1}\oplus \OO_{\PP^1}(2)\oplus \OO_{\PP^1}(3)$ at $t\neq 0$ 
as before, its projectivization, and take the 
closure of $X_{E}$ inside, which is our desired model. 
In this manner, we obtain further degeneration to surface of type $\D\A\cdots\A$. 

\end{Subcase}

\begin{Subcase}($\widetilde{\D_{16}}$ to $\D_{a}\A_{l_{1}-1}\cdots\A_{l_{m}-1}\D_{b}
$ with $a+b+l=16$)
Now we treat the case when the generic point of ${\rm Spec}K[[t]]$ 
maps to $M_{W}^{\rm nn}$ while the 
closed point maps to $M_{W}^{\rm seg}\cap M_{W}^{\rm nn}$, 
i.e., degenerations of $\widetilde{\D}_{16}$-type 
surfaces to type III surfaces. 

We lift the $K((t))$-rational point at $\mathcal{M}_{W}^{\rm nn}$ 
to $(g_{8}=3G_{4}^{2},g_{12}=G_{4}^{3})$ with $G_{4}\in H^{0}(\OO(4))$ 
with coefficient $K((t))$. By multiplying appropriate integer power of $t$, 
we can first assume that $G_{4}$ has all coefficients in $K[[t]]$. 
We also set the solutions of $G_{4}$ as $\sigma_{1},\sigma_{2},\tau_{1},
\tau_{2}$, 
which we can and do assume to be in $K((t))$ after finite base change of $K[[t]]$ 
if necessary. We suppose $\sigma_{i}|_{t=0}=0, \tau_{i}|_{t=0}=\infty$. 

Similarly to Case\ref{Case1} Step \ref{End.surfaces}, 
we set 

\begin{align}\label{f.def}
f(0):=\min \{{\rm val}_{t}(\sigma_{1}),{\rm val}_{t}(\sigma_{2})\},
\end{align}
\begin{align}\label{f'.def}
f(\infty):=\min \{{\rm val}_{t}(\tau_{1}^{-1}),{\rm val}_{t}(\tau_{2}^{-1})\}.
\end{align}
and  consider blow up $\mathcal{B}_{1}\to \mathcal{B}_{\rm triv}$ by 
\begin{align}
\langle s,t^{f(0)} \rangle \cdot \langle s^{-1}, t^{f(\infty)}\rangle.
\end{align}
Then, the special fibre of $\mathcal{B}_{1}$ over $t=0$ 
is 
\begin{align}\label{base.step1}
\mathbb{P}^{1}_{\frac{s}{t^{f(0)}}}\cup 
\mathbb{P}^{1}_{s}
\cup \mathbb{P}^{1}_{\frac{1}{s\cdot t^{f(\infty)}}}
\end{align} where the 
two ends are exceptional curves. 

Then, as in the Case\ref{Case1} Step \ref{End.surfaces}, 
the first component contains the limit of $\sigma_{i}|_{t\neq 0}$ 
and the last component contains the limiit of $\tau_{i}|_{t\neq 0}$ 
both different from the nodal points. 

Then similarly to 
we degenerate $\mathcal{O}_{\mathbb{P}^{1}}(2)\oplus \mathcal{O}_{\mathbb{P}^{1}}(3)$ on the original base to the whole model so that its reduction restricts to 
\begin{enumerate}
\item \label{1st..compo} 
$\mathbb{P}_{\mathbb{P}_{\frac{s}{t^{f(0)}}}^{1}}(\mathcal{O}_{\mathbb{P}^{1}}(2)\oplus \mathcal{O}_{\mathbb{P}^{1}}(3)\oplus \mathcal{O}_{\mathbb{P}^{1}})$ over $\mathbb{P}^{1}_{\frac{s}{t^{f(0)}}}$ 
\item \label{2nd..compo}
trivial $\mathbb{P}^{2}$-bundle 
over $\mathbb{P}^{1}_{s}$  (i.e., 
$\mathbb{P}^{2}\times \mathbb{P}^{1}_{s}$) 
\item \label{3rd..compo}
$\mathbb{P}_{\mathbb{P}_{s}^{1}}(\mathcal{O}_{\mathbb{P}^{1}}(2)\oplus \mathcal{O}_{\mathbb{P}^{1}}(3)\oplus \mathcal{O}_{\mathbb{P}^{1}})$ over $\mathbb{P}^{1}_{\frac{s'}{t^{f(\infty)}}}$. 
\end{enumerate}

Then our first step is to take closure of original $X$ inside the 
projectivization of the above $\mathbb{P}^{2}$-bundle on the rational chain. 

After this, we do the same procedures as 
 Step \ref{Step.middle}, Step \ref{Step.1.3} and then Step \ref{Real.end} 
 of Case \ref{Case1}. 
 Then we obtain the desired reduction to $\D\A\cdots\A\D$ type surface. 
\end{Subcase}

\end{Case}

By here, we complete the case by case reproof of stable reduction type 
Theorem~\ref{ABE.stable.reduction}. 
\end{proof}
Therefore, the completion of proof of Theorem~\ref{ABE.moduli} 
also follows the above (re)proof of Theorem~\ref{ABE.stable.reduction} 
(recall the beginning of our proof). 
\end{proof}

The identification of the 
normalization of $\overline{M_{W}^{\rm ABE}}$ 
with the toroidal compactification in \cite[\S7]{ABE} 
follows from the fact that the 
relative location of the real discriminants in the 
broken base chain of $\PP^{1}$s are encoded as 
$(\mathbb{G}_{m}\otimes \Lambda_{i})$. This 
may also follows again from further analysis in 
addition to above, but since this point overlaps more closely 
with the arguments in \cite{ABE} we do not pursue this here. See 
\cite[the proof of Proposition 7.45]{ABE}. 

Instead, we do some more explicit description. 

\begin{Cor}[of our reproof of Theorem~\ref{ABE.stable.reduction}]
\label{II.III.}
The boundary strata of $\overline{M_{W}}^{\rm ABE}$ which parametrizes 
degenerated surfaces of  the 
following stable types

\begin{align}
\begin{cases}
\E\A\E \\
\E\D \\ 
\E\A\cdots\A\D_{k} (\text{with } k\ge 9),
\end{cases} 
\end{align}
are not in the closure of two boundary prime divisors of Type II.

\end{Cor}

\begin{proof}
The first two starata are both $17$-dimension by the easy computation, 
while the Type II boundary components are also both $17$-dimension, 
hence the proof follows. The last stratum, the proof follows from 
our stable reducion arguments (or from the observation below). 
\end{proof}
We observe that, in our situation at least, 
if a surface component which corresponds to the lattice of $\Lambda$ type 
degenerates to those of type $\Lambda_{1}, \cdots, \Lambda_{m}$, 
$\Lambda_{1}\oplus\cdots\oplus \Lambda_{m}$ is a sublattice of $\Lambda$. This is partially explained in \cite{ABE} and also 
related to Proposition \ref{2lattices} to be explained. 

\begin{Rem}
Recall from \cite[\S 4.1]{DHT} combined with \cite[\S 3.3]{CD}, 
the interesting observation 
that one aspect of the classical Shioda-Inose structure 
construction to $II_{1,17}$-lattice polarized (higher Picard rank) 
K3 surface can be explained by an interesting Jacobian fibration 
which corresponds to the strata $M_{W}^{nn}$. 
The correspondence is explained via a part of Dolgachev-Nikulin 
mirror symmetry \cite[especially 7.11]{Dolg} i.e., the fiber of such Jacobian fibration 
plus the elliptic fiber of element of $M_{W}$ provides 
Type II degeneration from $M_{W}$ to $M_{W}^{\rm nn}$. 
This remark is not essentially new. 
\end{Rem}

\subsubsection*{Boundary strata of small codimensions}\label{bd.div.sec}
We classify boundary divisors and boundary strata of 
codimension $2$ of the  
compactification $\overline{M_{W}}^{\rm ABE}$. 
As prime divisors, there are at total $54$ of those as follows: 

\begin{enumerate}
\item 
$\E_{k_{1}}\A_{k_{2}}\E_{k_{3}}$ where 
$k_{1}+k_{2}+k_{3}=17, 
0\le k_{1}\le 8, 0\le k_{2}\le 17, 
0\le k_{3}\le 8.$ 
At total, we have $45$ boundary prime 
divisors of this type. 
The moduli is the product of Weyl group quotient of at total $17$-
dimensional algebraic tori (divided by left-right involution if $k_{1}=k_{3}$). 
\item 
$\E_{k}\D_{17-k}$ where $0\le k\le 8$. 9 of these boundary 
prime divisors. 
\end{enumerate}

The classification of 
$16$-dimensional boundary strata are as follows: 

\begin{enumerate}
\item 
$\E_{k_{1}}\A_{k_{2}}\A_{k_{3}}\E_{16-k_{1}-k_{2}-k_{3}}$ type with 
each $k_{i}\ge 0$. 
\item 
$\E_{k_{1}}\A_{k_{2}}\D_{16-k_{1}-k_{2}}$ with non-negative index. 
By \cite[\S 7I]{ABE}, the normalization 
$\overline{\M_{W}}^{\rm toroidal}\to 
\overline{M_{W}}^{\rm ABE}$ are nontrivial at the $9$ irreducible 
components of those with $k_{1}+k_{2}=16$, $0\le k_{1}\le 8$. 
\item 
$\D\D$ type. Again, by {\it loc.cit}, the normalizations are non
isomorphic at the one component for $\D_{0}\D_{16}$. 
\end{enumerate}

Hence, the normalization of $\overline{M_{W}}^{\rm ABE}$ 
are non-isomorphic at $9+1=10$ irreducible components of $16$-dimension (which is 
biggest dimension), and the preimage becomes $18+2=20$ components.



\part{Application to type II degeneration of K3 surfaces}

\section{Limit measure along type II degeneration}

\subsection{Limit points}

While the previous part I focuses on the {\it elliptic K3 surfaces}, 
their degenerations and moduli compactification, 
in this part II, we apply it to study more general K3 surfaces 
degeneration of type II over $\C$. 
The main point is, as in \cite{OO18}, 
the elliptic K3 structure appears around boundary as 
special Lagrangian fibration after suitable hyperK\"ahler rotation, 
as expected in the context of the Mirror symmetry and 
shown in \cite[\S 4]{OO18}. 
If we follow the setup of \cite[\S 6]{OO18}, 
we first observe the following.  

\begin{Lem}\label{II.lem}
If we naturally send $\mathcal{F}_{2d}\ni (X,L)$ 
into $\mathcal{M}_{K3}$ 
by adding $c_{1}(L)$ as additional period, 
type II cusps map to the strata $\mathcal{M}_{K3}(d)$ (see \cite[\S 6]{OO18}) 
of the Satake compactification of adjoint type 
$\overline{\mathcal{M}_{K3}}^{\rm Sat,adj}$. 
\end{Lem}

We refine the statements in Proposition \ref{lim.tau} which shows the limit existence in 
a yet another Satake compactification $\overline{\mathcal{M}_{K3}}^{\tau}$ among those non-adjoint types, 
which especially dominates the above compactification 
of adjoint type and 
dilates the $0$-dimensional locus $\mathcal{M}_{K3}(d)$ 
to $17$-dimension. 

\begin{proof}
As it is well-known, for type II degeneration, with some 
fixed marking, 
$\langle [{\rm Re}\Omega_{X}],[{\rm Im}\Omega_{X}]\rangle $ 
converge to isotropic plane while obviously $[\omega_{X}]$ 
remains the same class. Comparing with \S 6.2 of {\it loc.cit}, 
we obtain the proof. 
\end{proof}

Note that the locus $\mathcal{M}_{K3}(d)$ 
is nothing but the only $0$-cusp of the 
Satake-Baily-Borel compactification of $\mathcal{M}_{K3}(a)$, 
which is identified with the moduli of Weierstrass elliptic 
K3 surfaces modulo the involution (see \cite[\S 7]{OO18}). 
This is the key point to convert general problem on 
type II degeneration into type III degeneration of elliptic K3 
surfaces. In other words, roughly we divide the diverging 
isotropic plane into a line plus a line.

\subsection{Limit measure determination via  
Satake compactification}\label{measure.decide}

We now explicitly determine {\it measured} Gromov-Hausdorff limits (\cite{Fuk.mGH})
of tropical K3 surfaces in the sense of \cite[\S4]{OO18} 
so that we can justify the desired PL invariant $V$. 
That is, we study the collapse of 
$2$-dimensional spheres $S^{2}$ with the 
McLean metrics to unit intervals, through the algebro-geometric 
compactifications \cite{ABE} and its study 
in the previous section \ref{ABE.sec} of the asymptotic behaviour of 
singular fibers. 
This is an application of above stable reduction theorem after  
\cite{ABE}, 
providing one way of understanding of measured Gromov-Hausdorff limits 
classification (cf., \cite{Osh} for another way). 

We recall that Satake compactification of 
adjoint representation type coincides with 
certain generalization of Morgan-Shalen type compactification 
\cite[Theorem 2.1]{OO18}. This is the viewpoint we take in this section. 

For our purpose, we introduce the 
geometric realization map in a non-archimedean manner, 
which we write as $\tilde{\Phi}(a)$, as follows. 
This is essentially found by 
\cite[\S 4]{ABE} and 
Y.Oshima \cite{Osh} independently in somewhat different forms. 
The synchronization of the two works was 
rather surprising (at least to me) since their original 
aims were totally different, and also the tools are different: 
the latter was 
in more Hodge-theoritic 
context using a yet another Satake compactification as we define and briefly show below (see \cite{Osh} for details). 
No clear reason of the miraculous coincidence has been found yet, 
while our works mean to take a first step. 

\vspace{2mm}
\subsubsection*{Via a yet another Satake compactification}
As a preparation of precise statements, while more details 
are in \cite{Osh}, 
we consider the irreducible representation $\tau$ of 
$SO_{0}(3,19)$ whose 
highest root is only orthogonal to the leftmost one 
in the Dynkin diagram of \cite[\S 6.1]{OO18}. 
Then, as \cite{Osh} provides more details, 
the corresponding Satake compactification \cite{Sat1, Sat2} 
$\overline{\mathcal{M}_{\rm K3}}^{Sat, \tau}$ 
has $17$-dimensional strata 
$\mathcal{M}_{\rm K3}(d)^{\tau}$ which is 
$$O(\Lambda_{\rm seg})\backslash C^{+}(\Lambda_{\rm seg})/\R_{>0},$$ 
divided by the involution induced by complex conjugation. 
Here, $\Lambda_{\rm seg}:=p^{\perp}/p\simeq U\oplus E_{8}(-1)^{\oplus 2}$ 
with isotropic plane $p\subset \Lambda_{K3}\simeq U^{\oplus 3}
\oplus E_{8}(-1)^{\oplus 3}$, and 
\begin{align}\label{Cplus}
C^{+}(\Lambda_{\rm seg}):=\{x\in \Lambda_{\rm seg}\otimes \R\mid 
x^{2}>0\},
\end{align}
hence isomorphic to the $17$-dimensional real open unit ball. 
Its fundamental domain is provided by Vinberg's method (\cite{ABE, Osh}), 
and we here follow \cite[4C]{ABE} and denote as 
$P\simeq \mathcal{M}_{K3}(d)^{\tau}$ which is a subdivided 
Coxeter chamber (modulo the natural involution). $P$ is of the form: 
$P:=\{x\in C^+(\Lambda_{\rm seg})\mid (x,\alpha_i)>0\}$ for 
 in $\Lambda_{\rm seg}$. 

The formulation below, using Morgan-Shalen type compactification 
are re-designed to fit to the previous discussion of this paper. 

\begin{Def}[Geometric realization $\&$ measure density function]
\label{V.def}
We consider the quotient of 
\begin{align}\label{MSBJ..Sat}
\overline{\mathcal{M}_{W}}^{\rm MSBJ}
\simeq \overline{\mathcal{M}_{W}}^{\rm Sat,adj}
\end{align}
where the right hand side denotes the Satake compactification with respect to the adjoint representation of $SO_{0}(3,19)$ 
(the isomorphism is proven at \cite[2.1]{OO18} as a general theory) 
by 
$O(\Lambda_{\rm seg})/O^{+}(\Lambda_{\rm seg})$, acting as 
the complex conjugate 
involution. Then we obtain compactifications of 
$\mathcal{M}_{K3}(a)$ in \cite{OO18} respectively, 
which we denote as  
\begin{align}\label{MSBJ.Sat}
\overline{\mathcal{M}_{\rm K3}(a)}^{\rm MSBJ}
\simeq \overline{\mathcal{M}_{\rm K3}(a)}^{\rm Sat,adj}. 
\end{align}
Their common boundaries are hence stratified as follows: 
\begin{align}\label{domain.strata}
\mathcal{M}_{\rm K3}(a)\sqcup 
\mathcal{M}_{\rm K3}(d)^{\tau}
\sqcup \{2\text{ points } p_{seg} \text{ and } p_{nn}\}. 
\end{align}

Note this domain \eqref{domain.strata}, away from the two points 
$ p_{seg} \text{ and } p_{nn}$,  
 is also a subset of 
$\partial \overline{\mathcal{M}_{K3}}^{Sat, \tau}$. 
From the left hand side interpretation of \eqref{MSBJ.Sat}, 
$p_{seg}$ (resp., $p_{nn}$) corresponds to the prime divisor of 
toroidal compactifications over the $1$-cusp $M_{W}^{nn}$ 
(resp., $M_{W}^{seg}$) as \cite{CM05}. 

Now, we define \textit{geometric realization map} $\tilde{\Phi}$ 
from the above space \eqref{domain.strata} away from 
$p_{nn}$ to 
\begin{align*}
\{(X,d,\nu)\mid &\text{$(X,d)$ is a compact metric space with diameter one}\\
 &\qquad\qquad \text{and $\nu$ is a Radon measure}\}/\sim,
\end{align*}
where $\sim$ denotes the positive rescale of $\nu$, 
is defined after \cite[\S 7A]{ABE} as follows. 

\begin{enumerate}
\item For $x\in \mathcal{M}_{\rm K3}(a)$, 
we set $\tilde{\Phi}(x)$ as the tropical K3 surface 
$\Phi(x)$ as \cite[\S 6]{OO18} {\it with} its 
Monge-Ampere measure (equivalent to the volume form), 
as the (a priori) additional data. 

\item \label{open.part.realize}
(Open part of 
$P$: cf., \cite[\S 7A]{ABE} and \cite{Osh})
Recall that 
for each $l\in P\simeq 
O(\Lambda_{\rm seg})\backslash C^{+}(\Lambda_{\rm seg})/\R_{>0}$, which is neither $p_{seg}$ nor $p_{nn}$, 
\cite[\S 7A]{ABE} associates a polygon $P_{LR}(l)$ 
which can be rewritten as a translation of 
$$P_{LR}(l)=\{(x,y)\mid 0\le x\le 1, 0\le y\le (V(l))(x)\},$$ 
for some PL function $V(l)$. 
Then we set $\tilde{\Phi}(l)$ 
as $[0,1]$ with the density function $V(l)$. 

\item \label{pseg.realize}
(A special point  $p_{seg}$ cf., \cite{Osh}) 
We set $\tilde{\Phi}(p_{seg}):=([0,1],d,\nu)$ 
with standart metric $d$ and $\nu\equiv 0$. 

\end{enumerate}
\end{Def}

\begin{Thm}[cf. also \cite{Osh} for another proof]\label{continuity2}
The geometric realization map $\tilde{\Phi}$ 
is continuous 
with respect to the measured Gromov-Hausdorff topology in the sense of 
\cite{Fuk.mGH}. 
\end{Thm}
\noindent
As we mentioned, \cite{Osh} gives a different proof for this, 
notably Steps \ref{st.red.per}, \ref{ABE.V}.  
\begin{proof}

First, we fix a notation and make a setup: we take a 
sequence of $(g_{8},g_{12})$ with subindex $i$ whose 
Weierstrass models in $M_{W}$ which converge to 
a point in $M_{W}^{seg}$, the union of a $1$-cusp and the $0$-cusp,  
in the Satake-Baily-Borel compactification $\overline{M_{W}}$. 
Recall that we show it is isomorphic to the GIT quotient 
compactificatoin of $M_{W}$ 
with 
respect to the Weierstrass model description in \cite[Theorem 7.9]
{OO18}. 
Taking $(c_{1}s^{4},c_{2}s^{6})$ as a GIT polystable 
representative of the limit point in $M_{W}^{seg}$, 
by the Luna slice \'etale theorem at stacky level 
(cf., \cite{Luna, Drezet})
for instance, 
we can and do assume our sequence of $(g_{8},g_{12})$ 
converges to it. For later use, for each 
$i$, we consider the roots of $g_{8}$ (resp., $g_{12}$, 
$\Delta_{24}$) and denote as 
$\{\xi_{j}\}_{j=1,\cdots,8}$ 
(resp., $\{\eta_{j}\}_{j=1,\cdots,12}$, 
$\{\chi_{j}\}_{j=1,\cdots,24}$) 
in ascending order of the absolute values. 
The natural analogues of $e(0)$ \eqref{e.def} and $e(\infty)$ \eqref{e'.def} in our stable reduction arguments 
i.e., sequence version are 
\begin{align}\label{e.e'.seq}
&\epsilon:=\max\{|\xi_{j}|, |\eta_{j'}|
 \mid 2\leq j\leq 4,\ 1\leq j'\leq 6\},\\
&\epsilon':=\max\{|\xi_{j}|^{-1}, |\eta_{j'}|^{-1}
 \mid 5\leq j\leq 7,\ 7\leq j'\leq 12\}. 
\end{align}

\begin{Step}\label{Step1.}
Firstly, this Step \ref{Step1.} focuses on the case when the sequence of 
$[(g_{8},g_{12})]$ converges to a point 
in the $1$-cusp, $M_{W}^{seg}\setminus M_{W}^{nn}$. 

In this case, 
\cite[\S 7.3.2]{OO18} shows 
the corresponding sequence of McLean metrics
 converges to 
infinitely long open surface which is asymptotically cylindrcical 
at two ends $0$ and $\infty$, 
as minimal non-collapsing 
pointed Gromov-Hausdorff limit.  

In this case, for large enough $i$ i.e., 
with the McLean metric close enough to the above 
asymptotically cyrindrical 
surface,  
\cite[\S 7.3.7, notably Lemma 7.26]{OO18} 
implies the following: after rescale with fixed diameters, 
in particular with bounded above distance of $s=0$ and 
$s=\infty$, the corresponding renormalized $\rho(r)$ in 
{\it loc.cit} uniformly converges to 
$0$ (after making $r$ bounded by rescale) so that even the full 
measure of the (rescaled) McLean metric also 
tends to $0$ for $i\to \infty$. 

Hence we obtain desired convergence to the interval with 
$0$ measure, as metric measure space, in the 
sense of e.g. \cite{Fuk.mGH}.

\end{Step}

\begin{Step}\label{Step2.}
This Step \ref{Step2.} provides the first step analysis of the ``maximally degenerate'' case 
when $c_{1}=3c_{2}=3$, and is borrowed from \cite{Osh}, 
which we follow and leave for the proof. (Our later steps  
are different from \cite{Osh}, with more algebro-geometric or 
non-archimedean perspectives.) 
We thank Y.Oshima for the permission to write also here. 
For each $i$, we define a cut-off function on $\R_{>0}$ as 
\[\varphi(r):=\begin{cases}
 -1 & r<\epsilon, \\
 \frac{\log r}{|\log \epsilon|} & \epsilon\leq r\leq \epsilon'^{-1}, \\
 1 & r >\epsilon'^{-1}\end{cases}
\]
Here, for each $j$, suppose  $\lim_{i\to \infty}
\varphi(|\chi_{j}|)=:x_j$ 
(the appearance of two indices $i, j$ are not typo as 
$j$ is fixed here while $\chi_{j}$ depends on $i$) 
which is negative for $j\le k$ and non-negative otherwise. 
In addition, we may assume that
 $- \frac{\log |D_{i}|}{|\log \epsilon_i|}\to d\in [0, +\infty]$, 
 where $D_{i}$ denotes the top coefficient of $\Delta_{24}$. 
Then \cite{Osh} determines the limit measure on 
the interval by using the approximate description of the 
McLean metric \cite[\S7.3.3, notably Lemma 7.16]{OO18}. 
The limit measure can be described as (up to positive 
constants multiplication) 
$V$ on $[-1,1]$ by 
\[V(w) = 12 w + d
 - \sum_{j=1}^k \max\{w, x_j\}
 - \sum_{j=k+1}^{24} \max\{0, w - x_j\}.
\]
and as in \cite{HSZ}, metric $d$ and measure $\nu$ on the interval 
$[-1,1]$ as 
\begin{align*}
&d=V(w)^{\frac{1}{2}} dw, \quad \nu=V(w)dw, \quad
 \text{if $V\not\equiv 0, +\infty$,} \\
&d=dw,\quad \nu=dw \quad \text{if $V\equiv 0$ (or $V\equiv +\infty$}).
\end{align*}
\begin{Lem}[\cite{Osh}, compare with \cite{HSZ}]\label{lim.measure.Osh}
For the given and fixed sequence of $(g_{8},g_{12})$, 
the underlying base $\PP^{1}$ with McLean metric of the 
Weierstrass elliptic K3 surfaces converges to 
the above $([-1,1],d,\nu)$ as the metric measure space, 
up to rescale. 
\end{Lem}
\end{Step}

\begin{Step}\label{st.red.per}

We consider the normalized compact moduli $\overline{M_W}^{\rm ABE,\nu}$ 
and its stacky refinement $\overline{\mathcal{M}_W}^{\rm ABE,\nu}$ (a proper Deligne-Mumford algebraic stack) 
which comes from the construction of 
$\overline{\mathcal{M}_W}^{\rm ABE}$
in \S \ref{ABE.reconst.sec}  i.e.,  by the log KSBA moduli interpretation after \cite{ABE}. 

Take an \'etale chart $\overline{\mathcal{U}}$ of the stack $\overline{\mathcal{M}_W}^{\rm ABE,\nu}$ 
which contains the preimage of the 
$0$-cusp of $\overline{M_W}$. We denote the preimage of the open part $M_W$ as $\mathcal{U}\subset \overline{\mathcal{U}}$. 
Denote the corresponding coarse moduli as $U\subset \overline{U}$. 
Now we apply the Morgan-Shalen compactification as 
\cite[Appendix]{TGC.II} to $\mathcal{U}\subset \overline{\mathcal{U}}$ 
and denote it simply as $U\subset \overline{U}^{\rm MSBJ}$. 

As preparation, now we define the following modified Newton polygon of the discriminant $\Delta_{24}$ 
for a sequence of $(g_8,g_{12})$ with respect to $i=1,2,\cdots$ converging to $(3s^4,s^6)$. 
For $\Delta_{24}(s)=\sum_{j=1}^{24} d_j s^j$, 
we set 
\begin{align*}\label{arch.d.hull}
{\rm Newt}(\Delta_{24}):=\{(j,-\log|d_{j}|)\mid 0\le j\le 24\}+\R_{\ge 0}(0,1) 
\end{align*}
as an analogue of \eqref{d.hull} and modify it by using 
$\epsilon$, $\epsilon'$ of 
\eqref{e.e'.seq}, a sequence analogue of $e(0), e(\infty)$ (during the proof of Claim \ref{Max.degen.stable.reduction}),  
as follows: 
first we regard the above ${\rm Newt}(\Delta_{24})$ as a graph of PL convex function 
$\varphi^{\C}_{\Delta}\colon [0,24]\to \R\cup \{\infty\}$ and modification is defined below. 
We set similarly as before 
\begin{align}
i_{\epsilon}:=\max\{i\mid \varphi_{\Delta}(i)-\varphi_{\Delta}(i+1)\ge 
\epsilon\},
\end{align}
\begin{align}
i_{\epsilon'}:=\min\{i\mid \varphi_{\Delta}(i+1)-\varphi_{\Delta}(i)\ge 
\epsilon'\}. 
\end{align}
Again as before, we modify $\varphi^{\C}_{\Delta}$ to 
$\overline{\varphi}^{\C}_{\Delta}\colon [0, 24]\to \R\cup \{\infty\}$ as 
\begin{align}\label{PL.modif}
\bar{\varphi}^{\C}_{\Delta}(i):=
\begin{cases}
\varphi_{\Delta}^{\C}(i_{\epsilon})-\epsilon(i_{\epsilon}-i) & (\text{if }0\le i\le i_{\epsilon})  \\ 
\varphi_{\Delta}^{\C}(i) & (\text{if }i_{\epsilon}\le i\le i_{\epsilon'})   \\
\varphi_{\Delta}^{\C}(i_{\epsilon'})+\epsilon'(i-i_{\epsilon'}) & (\text{if }i_{\epsilon'}\le i\le 24). 
\end{cases}
\end{align}
We are actually only concerned about it modulo positive constant 
multiplication, but anyhow denote the graph of $\bar{\varphi}^{\C}
_{\Delta}$ as ${\rm  Newt'}(\Delta_{24})$. Note that, from the 
definition using 
the (archimedean) logarithm, 
the non-differentiable points in the domain is not necessarily 
integers. For instance, along any 
holomorphic punctured family of $(g_{8},g_{12})$ 
converging to $(3s^{4},s^{6})$, 
the obtained limit of the above ${\rm  Newt'}(\Delta_{24})$ 
modulo rescale (fixing the height) becomes 
our 
$P_{\Delta,0}$ in \eqref{P.Delta}, the epigraph of 
$\overline{\varphi}_{\Delta}$ in \eqref{PL.modif}. 
We can and do assume our sequence sits in a neighborhood 
$U''$ 
of $((3,\vec{0});(1,\vec{0}))$ 
in $\A^{22}=\A^{9}\times \A^{13}$ describing the coefficients of $g_{8}$s and $g_{12}$s
for each $i$. 
We consider the rational map from $U''$ to 
some (arbitrarily fixed) toroidal compactification 
$\overline{M_{W}}^{AMRT, \{\Sigma\}}$ and 
replace $U''$ by its blow up to make it a morphism. 
We denote the preimage of the boundary as $D''\subset U''$, 
and set $U''':=U''\setminus D''$. 

Now, 
we apply the functoriality of MSBJ construction 
\cite[Appendix \S A.2, A.15]{TGC.II} (more precisely, 
the analytic extension in 
\cite{SBB}), we obtain a continuous map 
$\overline{U'''}^{\rm MSBJ}(U'')\to \overline{M_{W}}^{\rm MSBJ}
(\overline{M_{W}}^{AMRT, \{\Sigma\}})$. 

From the previous Step \ref{Step2.}, 
the limit of ${\rm  Newt'}(\Delta_{24})$ for $i\to \infty$ 
decides the measured Gromov-Hausdorff limit of McLean metrics sequence 
\eqref{lim.measure.Osh}, which is metrically the interval. 
Thus, from the case-by-case proof of Claim 
\ref{Max.degen.stable.reduction} during that of 
Theorem \ref{ABE.stable.reduction}, 
above discussion readily implies that: 
\begin{Claim}
The measured Gromov-Hausdorff limit of McLean metrics sequence 
\eqref{lim.measure.Osh} is 
determined by the limit point inside the Morgan-Shalen type 
compactification 
$\overline{U'''}^{\rm MSBJ}(U'')$ (if exists).  
\end{Claim}
\end{Step}


\begin{Step}\label{ABE.V}
If we consider the set of points of the boundary 
$\partial \overline{U'''}^{\rm MSBJ}(U'')$, whose (given integral) 
affine coordinates valued in $\Q$, it is obviously dense. 
On the other hand, 
recall from the previous Step \ref{st.red.per} 
that there is a natural continuous map 
$\overline{U'''}^{\rm MSBJ}(U'')\to \overline{M_{W}}^{\rm MSBJ}
(\overline{M_{W}}^{AMRT, \{\Sigma\}})$. 
Hence, it is enough to show the following claim: 

\begin{Claim}\label{V.determine}
For any point $p \in 
\partial \overline{U'''}^{\rm MSBJ}(U'')$ with rational affine 
coordinates, 
if we describe its image in 
$\overline{M_{W}}^{\rm MSBJ}
(\overline{M_{W}}^{AMRT, \{\Sigma\}})$ 
as $\bar{l}:=\R l$ with 
$(0 \neq) l=l(p)\in C^{+}(\Lambda_{\rm seg})\cap 
\Lambda_{\rm seg}\otimes \Q$ (we also denote $\bar{l}=\overline{l(p)}$) 
the limit measure density function $V$ (\cite{HSZ}, our previous 
Step \ref{Step2.}) for some sequence in 
$M_{W}$ 
converging to $p$, 
coincides with $\tilde{\Phi}(\overline{l(p)})$. 
\end{Claim}

To prove the Claim \ref{V.determine}, recall that \cite[Theorem 1.2]{ABE} shows that the 
normalization of the log KSBA compactification of the 
Weierstrass elliptic K3 surfaces with their ``(weighted) rational 
curves cycle'' 
type boundaries is the toroidal compactification (\cite{AMRT}) 
with respect to the rational curves cone. 
As its first step, they construct, 
for given $(0 \neq) l\in C^{+}(\Lambda_{\rm seg})\cap 
\Lambda_{\rm seg}\otimes \Q$, 
a certain Kulikov model $X_{LR}(l)$ (and its flop 
$X'_{LR}(l)$ after a base change). 
For $l$, 
we take such models as the one in Claim \ref{V.determine}. And 
one can assume the image of $t$ in $\Delta^{*}$ 
converges to $p$ for $t\to 0$. Indeed, 
we can take $X_{LR}(l)$ to be the 
pull back of the Kulikov (semistable) model family, 
constructed 
in \cite{ABE}, to a generic analytic curve 
transversally intersecting 
the open strata of the prime divisor of $U''$ corresponding to $p$ 
(if such divisor does not exist, we simply replace $U''$ by 
blow up satisfying it). 
Then {\it loc.cit} showed 
that its monodromy invariant (cf., e.g., \cite{FriS}) 
is nothing but $l$ modulo $O(\Lambda_{\rm seg})$ 
in Corollary 7.33 {\it loc.cit}. It is done 
using the crucial diffeomorphism from degenerating elliptic 
K3 surface to  
a corresponding Symington type 
Lagrangian fibration by bare hand \cite{EF} 
and then calculating the intersection numbers on 
the Lagrangian fibration side. 
Recall from \cite[Theorem 2.8, Corollary 4.25]{OO18} that 
the limit inside MSBJ compactification 
$\overline{M_{W}}^{\rm MSBJ}
(\overline{M_{W}}^{AMRT, \{\Sigma\}})$
is equivalent to the information of the monodromy on 
$U^{\perp}$ of signature $(2,18)$. 

For each $X_{LR}(l)$ as above, one can directly see the limit 
measure density function by our previous Steps combined with 
the case-by-case explicit 
proof of Claim \ref{Max.degen.stable.reduction}, 
and coincides with 
$\tilde{\Phi}(\bar{l})$ which is determined by the monodromy. 
Hence, it is determined by the limit inside $\overline{M_{W}}^{\rm MSBJ}
(\overline{M_{W}}^{AMRT, \{\Sigma\}})$ 
by \cite[Theorem 2.8, Corollary 4.25]{OO18} and 
the claim \ref{V.determine} for general sequence, 
the desired coincidence (Theorem \ref{continuity2}) finally follows. 
\end{Step}
\end{proof}

\subsection{Explicit description and examples}

Recall that, in particular, 
$\tilde{\Phi}(l)$ in case \eqref{open.part.realize} of 
Definition \ref{V.def} is as follows,   
as \cite[\S 7A]{ABE}, \cite{Osh}, which describes all the 
details from which we borrow.  
The fundamental polygon $P$ is divided into $9=3^{1+1}$ 
maximal chambers, say $\{P'_{a}\}_{a}$, and 
the points of $[0,1]$ where $(\tilde{\Phi}(l))(0)$ is non-differentiable 
can be written as 
\begin{align}
0=\frac{q_{-2}}{q_{22}}= \frac{q_{-1}}{q_{22}}= \frac{q_{0}}{q_{22}}
\le 
\frac{q_{1}}{q_{22}}\le \cdots \le 
\frac{q_{19}}{q_{22}}
\le \frac{q_{20}=q_{21}=q_{22}}{q_{22}}=1.
\end{align}

The definitions also imply 
\begin{align}
\frac{q_{1}}{q_{22}}
=\max \biggl\{(l,-\frac{1}{3}\beta_{L}), 0\biggr\},
\end{align}
with $\beta_{L}\in \Lambda_{\rm seg}$ (see \cite[\S 4C]{ABE}) and 
every $q_{j}$ are linear at each $P'_{a}$ with respect to the 
description \eqref{Cplus}. 

The values and slopes of the function satisfy 
\begin{align}
(\tilde{\Phi}(l))(0)=\max \{(l,\beta_{L}),0\},
\end{align}
\begin{align}
\dfrac{d \tilde{\Phi}(l) (x)}{dx}=9-i \text{ for any } x \in 
(\frac{q_{i}}{q_{22}},\frac{q_{i+1}}{q_{22}}). 
\end{align}

In particular, $\tilde{\Phi}(l)$ is convex. 
Indeed: 
\begin{itemize}
\item 
if $(l,\beta_{L})\le 0$, for generic $l$ under such assumption, 
$\tilde{\Phi}(l)(0)=0$ and 
the slope of $\tilde{\Phi}(l)$ starts with $9$ and decrease by $1$ at 
each wall crossing through $q_{j}$. 
\item 
if $(l,\beta_{L})\ge 0$, then for generic $l$ under such assumption, 
the slope of $\tilde{\Phi}(l)$ starts with $8$ and decrease by $1$ at 
each wall crossing through $q_{j}$. 
\end{itemize}
In the case $\tilde{\Phi}(l)(0)=\tilde{\Phi}(l)(1)=0$ (e.g., \S \ref{HSVZ.sec}), 
then note that the barycenter of 
$q_{i}$ is the middle point $\frac{1}{2}$. 
The behaviour of the function 
$(\tilde{\Phi}(l))$ around the opposite end $1$ (denoted by $R$ in 
\cite{ABE}) is completely similar.

\begin{Rem}[{Relation with \cite[\S 5]{CM05}}]
For one parameter Type III degenerations from $M_{W}$ to the 
locus inside the closure of $M_{W}^{\rm nn}$, 
we expect that 
the corresponding limit point in $M_{W}(d)^{\tau}$ 
can be explained by the collision of $18$ 
blow up centers $p_{i}$s for the stable type II degeneration 
of those elliptic K3 surfaces introduced in \cite[\S 5]{CM05}. 
For the combinatorial type of such type III degenerations, recall 
Corollary~\ref{II.III.}. 
\end{Rem}

\begin{Ex}[Via Davenport-Stothers triple]\label{Zensha.ex}

Here we see simple examples of degenerating Weierstrass elliptic K3 
surfaces 
and apply above to obtain the limit measures of the family of 
McLean metrized spheres. 

In the following two examples, let us denote 
\begin{align}
g_{4}(s):=3(s^{4}+2s),
\end{align}
\begin{align}
g_{6}(s):=s^{6}+3s^{2}+\frac{3}{2},
\end{align}
so that
\begin{align}
g_{4}^{3}-27g_{6}^{2}=-27(s^{3}+\frac{9}{4}).
\end{align}
Up to affine transformation, this is known to be the 
only pair of degree $4$, degree $6$ polynomials 
with the degree of $g_{4}^{3}-27g_{6}^{2}$ is $3$. 
It is an easy example of ``Davenport-Stothers triple'' 
(cf., e.g., \cite{Dav},\cite{Sto},\cite{Zan},\cite{Shi}).

Our first example is as follows: 
\begin{align}
g_8(s):=g_4\biggl(\frac{s}{t}\biggr)g_4\biggl(\frac{1}{ts}\biggr)s^4,
\end{align}
\begin{align}
g_{12}(s):=g_{6}\biggl(\frac{s}{t}\biggr)
g_{6}\biggl(\frac{1}{ts}\biggr)s^6
\end{align}
for $t\to 0.$
Then we see that the density function $V$ of the limit 
measure of the tropical K3 surfaces is 
as follows (modulo rescale): 

\begin{align}
V(a)= \left\{
\begin{array}{ll}
a & 0\le a\le \frac{1}{2}\\
1-a & \frac{1}{2}\le a\le 1,\\
\end{array}
\right.
\end{align}
which is directly checkable after our arguments in \S 
\ref{ABE.reconst.sec} 
and \cite[\S 7A]{ABE}. 
\end{Ex}

\begin{Ex}[Via Davenport-Stothers triple again]\label{Kousha.ex}

We use the same $g_{4}, g_{6}$ as above Ex~\ref{Zensha.ex} 
while construct different $g_{8}, g_{12}$s. 
Note
$$\biggl(s+\frac{1}{s}\biggr)^{-1}=\frac{s}{(s^2+1)},$$
$(s+\frac{1}{s})^{-1}$ is near $0$ if and only if  $s$ is near $0$ or $\infty$. 
Thus $(s+\frac{1}{s})^{-1}$ is near $\infty$ if and only if  $s$ is near 
$\sqrt{-1}$. 
In this example, we define $g_{8}, g_{12}$ as follows: 
$$g_8(s):=g_4\biggl(\frac{s}{t(s^2+1)}\biggr)\cdot (s^2+1)^4,$$
$$g_{12}(s):=g_6\biggl(\frac{s}{t(s^2+1)}\biggr)\cdot (s^2+1)^6.$$

Then $$\Delta_{24}(s)=g_{8}^{3}-27 g_{12}^{2}=0 \in 
\mathcal{O}_{\mathbb{P}^{1}}(24)|_{s}$$
if and only if 
$$\frac{s}{t(s^2+1)}  = \chi_i (i=1,2,3)$$
or $$\frac{s}{(s^2+1)} = \infty$$ with multiplicity $18$
if and only if  
\begin{align}\label{eqn}
s+\frac{1}{s} = (t\chi_i) ^{-1} (i=1,2,3)
\end{align}
or 
\begin{align}\label{eqn2}
s+\frac{1}{s} = 0
\end{align}
where, the latter with multiplicities $18$. 
The former \eqref{eqn} happens 
if and only if  
$$s=\frac{1\pm \sqrt{1-4t^{2}\chi_{i}^{2}}}{2}$$ 
and the latter happens when either $s=\sqrt{-1}$ with the multiplicity $9$ or $s=-\sqrt{-1}$ with the multiplicity $9$ again. 
Therefore, if we $t\to 0$, we get $[0,1]$ with the corresponding $V$ (modulo rescale) as 
same again: 

$$V(a)= \left\{
\begin{array}{ll}
a & 0\le a\le \frac{1}{2}\\
1-a & \frac{1}{2}\le a\le 1.\\
\end{array}
\right.
$$
\end{Ex}
In next \S \ref{HSVZ.sec}, we observe that above 
two cases are 
close to the direction of collapsing of \cite{HSVZ}. 

\begin{Ex}[Simplest D type]
On the other hand, as another simple our instance of our 
discussion in the proof of Theorem \ref{ABE.stable.reduction}, 
we obtain a different type of $V$ with $V(0)=V(1)\neq 0$. 

Set 
\begin{align*}
& g_{8}(s)=3(
(s-ta_{1})(s-ta_{2})
(ts-a_{3})(ts-a_{4})
)^{2}, \\ 
& g_{12}(s)=(
(s-ta_{1})(s-ta_{2})
(ts-a_{3})(ts-a_{4})
)^{3},\\ 
\end{align*}
for $a_{1}\neq a_{2}$, $a_{3}\neq a_{4}$, all lie in $K$. 
Then, the Newton polygon of $\Delta_{24}$ has only two 
slopes, so that the proof (Case 1, 2) of 
Theorem \ref{ABE.stable.reduction} 
shows the corresponding $V$ for $t \to 0$ is 
a constant function. 

Indeed, this is the simplest prototypical example of 
D type degeneration of elliptic K3 surfaces. 
\end{Ex}
From the definition \ref{V.def} of our $\tilde{\Phi}$, 
and compare with \cite[\S 7A]{ABE} or \cite{Osh}, 
Theorem \ref{ABE.stable.reduction} ensures that 
$V$ can have much more varieties in general. 



\section{Limits along Type II degeneration and associated lattices}\label{Alg.limits.sec}

As claimed in our introduction, 
we are now ready to give general considerations on 
limits along $\mathcal{F}_{2d}$ to make sense of 
the $V$ function for type II degenerations. 
Suppose we are in the repeated setup as 
$(\mathcal{X},\mathcal{L})\to \Delta$ in $\mathcal{F}_{2d}$ 
is a type II polarized degeneration of K3 surfaces, 
dominated by a Kulikov model $\tilde{\mathcal{X}}$ and the 
pull back $\tilde{\mathcal{L}}$ of $\mathcal{L}$ to 
$\tilde{\mathcal{X}}$, and a stable type II degeneration 
$\mathcal{X}_0=V_0\cup V_1$. Then, refining Lemma \ref{II.lem}, the following holds. 
\begin{Prop}\label{lim.tau}
For the given $\pi\colon (\X,\mathcal{L})\to \Delta$ as above, 
the naturally associated continuous map 
$\varphi^{o}$ from $\Delta\setminus 0$ to 
$\mathcal{M}_{K3}$ continuously extends to a map 
$\varphi$ from 
$\Delta$ with $\varphi(0)=c(\X,\mathcal{L})$
in $\mathcal{M}_{\rm K3}(d)^{\tau}
\subset \overline{\mathcal{M}_{\rm K3}}^{\rm Sat,\tau}$. 
In other words, the limit point inside 
$\overline{\mathcal{M}_{\rm K3}}^{\rm Sat,\tau}$ 
for $t\to 0$ is well-defined. In particular, 
there is the well-defined function $V=V_{\pi}=V(\X,\mathcal{L}):=
\tilde{\Phi}(c(\X,\mathcal{L}))$ on the segment 
for this $(\X,\mathcal{L})$ as we noted in the beginning of the paper. 
\end{Prop}

\begin{proof}
The proof is easy as Lemma \ref{II.lem}, as through 
a marking $\alpha$ of $H^{2}(\X_{1},\Z)$, 
$\varphi^{o}(t)$ clearly converges to the 
image of the K\"ahler class $\alpha(c_{1}(\mathcal{L}|_{\mathcal{X}
_{1}}))$ 
for $t\to 0$. 
\end{proof}
We remark that in the collaboration with Oshima, 
the above limit is expected to describe the limit measure and more 
generally 
$\tilde{\Phi}$ to be continuous on 
whole $\mathcal{M}_{K3}\sqcup \mathcal{M}_{\rm K3}
(d)^{\tau}
(\subset \overline{\mathcal{M}_{\rm K3}}^{\rm Sat,\tau})$ with respect to the 
{\it measured} Gromov-Hausdorff topology so that the 
above $V_{\pi}$ determines the limit measure of the hyperK\"ahler metrics on general fibers. \cite{Osh} provides 
related discussions. 

Furthermore, we take a marking 
$H^{2}(\mathcal{X}_{1},\Z)\simeq \Lambda_{\rm K3}$ 
so that the corresponding isotropic plane is 
$$\Z e''\oplus \Z e'$$
and we denote the image of $c_{1}(\mathcal{L}|_{\X_{1}})$ as 
$v_{2d}$ of norm $2d$. 
Recall the canonical isomorphism 
$$\langle e'', e' \rangle ^{\perp}/\langle e'', e' \rangle
\simeq \Lambda_{\rm seg}=II_{1,17}
\simeq U\oplus E_{8}^{\oplus 2}.$$
We write $\langle e'', e'\rangle =:p$. 
Then, $v_{2d}^{\perp} \subset p^{\perp}/p$ is 
studied classically in e.g. \cite{Fri84}, 
which we denote as 
$\Lambda_{\rm per}(c)=\Lambda_{\rm per}(c(\X,\mathcal{L}))$. 

\vspace{2mm}
As a hyperK\"ahler rotated side, 
we take a type III degeneration $\mathcal{X}^{\vee}\to \Delta$ of 
Weierstrass elliptic K3 surfaces which 
we suppose to be Kulikov degeneration, 
i.e., $(\mathcal{X}^{\vee},\mathcal{X}^{\vee}_{0})$ is 
log smooth and is minimal. 
We put a marking on the smooth fibers so that the 
elliptic fiber class is $e''$ and the zero-section class is 
$f''$. Recall that from 
\cite[\S7]{ABE}, 
an irreducible decomposition of 
$\mathcal{X}^{\vee}_{0}$ which we write as 
$\cup_{i}V^{\vee}_{i}$ satisfies each $V_{i}$ (or its pair) 
are either of the following forms: 
\begin{itemize}
\item $XI\cdots IX$, 
\item $Y_{2}Y_{a}I\cdots IX$, 
\item $Y_{2}Y_{a}I\cdots IY_{2}Y_{a}$. 
\end{itemize}
$\mathcal{X}^{\vee}\to \Delta$
It is easy to confirm that 
after appropriate flops, we can and do assume that 
the non-toric component (i.e., those with positive charges) 
all remains at the stable model of \cite{ABE}. 
Then, such remaining 
rational surfaces $V_{i}$ 
with normal crossing boundary $\cup_{j}D_{i,j}$ 
and are encoded as slightly generalized root lattice 
of type either $\D\A\cdots\A\D$, $\D\A\cdots\A\E$, 
$\E\A\cdots\A\E$ with possibly indices $0$s. 
This is encoded in {\it loc.cit} 
as $P_{LR}(l)$ (resp., piecewise linear function $V$). 
We denote such lattice as $\Lambda_{\rm ABE}(\mathcal{X}^{\vee})$. 
Note that its rank is generally $0$ and at most $17$. 
On the other hand, as a hyperK\"ahler rotation 
of $(\X^{\vee}_{t},\omega^{\vee}_{t})$ 
with $[\omega^{\vee}_{t}]= m_{t}e''+f''$ with $m_{t}\to \infty$, 
we set $\{(\X_{t},\omega_{t})\}_{t}$ of type II 
for $t\to 0$ (as in \cite[\S 4]{OO18}). 
We anyhow denote the limit 
inside the Satake compactification $\overline{\mathcal{M}_{K3}}
^{\rm Sat,\tau}$ formally as $c(\X,\mathcal{L})$. 
Then, the following holds. 

\begin{Prop}\label{2lattices}
In the above setup, the two associated negative definite lattices has 
canonical inclusion which respects the bilinear forms: 
$$\Lambda_{\rm ABE}(\mathcal{X}^{\vee})
\subset \Lambda_{\rm per}(c(\X,\mathcal{L})).$$
\end{Prop}

\begin{proof}
Recall that the $\Lambda_{\rm ABE}(\mathcal{X}^{\vee})$ 
(\cite[\S 7G, \S7H]{ABE})
is the direct sum of the slightly generalized ADE lattices
$(\sum_{j}\Z [D_{i,j}])^{\perp}\subset H^{2}(V_{i},\Z)$. 
We use Clemens 
contraction map $\mathcal{X}^{\vee}_{1}\to \mathcal{X}^{\vee}_{0}$, 
and the marking of $\mathcal{X}^{\vee}_{1}$ so that 
we can regard $H^{2}(\mathcal{X}^{\vee}_{0},\Z)$ canonically 
\footnote{modulo the monodromy, 
but the classes in our actual concern are all monodromy invariant and 
further if one fixes a continuous 
path connecting $0$ and $1$ in $\Delta$, 
then it becomes canonical. }
as a sublattice of $\Lambda_{K3}$. 

Any $(\sum_{j}\Z [D_{i,j}])^{\perp}\subset H^{2}(V_{i},\Z)$ lies in 
$(1,1)$-part. On the other hand, 
from the construction of the 
 hyperK\"ahler rotation $\mathcal{X}^{\vee}$, one of its period (
 real part of the cohomology of the holomorphic volume form) 
 converges to $v_{2d}$ as $(2,0)$-part. Hence they are orthogonal. 
This completes the proof. 
\end{proof}

\begin{Ex}
If $2d=4$, i.e., degenerations of quartics, there are 
certainly examples where the above two lattices 
$\Lambda_{\rm ABE}(\mathcal{X}^{\vee})$ and 
$\Lambda_{\rm per}(\X,\mathcal{L})$ 
do not coincide. For instance, 
if $v_{2d}=2e''+f''$, then 
$$\Lambda_{\rm ABE}(\mathcal{X}^{\vee})\simeq E_{8}(-1)^{\oplus 2}$$
while
$$\Lambda_{\rm per}(\X,\mathcal{L})\simeq 
\langle -4 \rangle \oplus E_{8}(-1)^{\oplus 2}.$$
Also, there is another example with $2d=4$ 
such that $$\Lambda_{\rm ABE}(\mathcal{X}^{\vee})\simeq D_{8}(-1)^{\oplus 2}$$
while
$$\Lambda_{\rm per}(\X,\mathcal{L})\simeq 
\langle -4 \rangle \oplus D_{8}(-1)^{\oplus 2}.$$
\end{Ex}

\vspace{2mm}
\begin{Rem}
Similar even negative definite 
lattices appear also in a slightly different context of 
Dolgachev-Nikulin mirror symmetry for 
lattice polarized K3 surfaces \cite{Dolg}. 
Recall that the Dolgachev-Nikulin 
mirror (\cite[7.11]{Dolg}, \cite[4.1]{DHT}) 
of $\mathcal{F}_{2d}$ says, to each type II degeneration in 
$\mathcal{F}_{2d}$, there is 
an associated isotropic element $e(\X,\mathcal{L})$ 
in $\Lambda_{2d}$ modulo $\tilde{O}(\Lambda_{2d})$. 

From the arguments in \cite[4.14, 4.18, 6.10]{OO18}, 
in an open neighborhood of $0$-cusp, 
$e(\X,\mathcal{L})$ induces 
elliptic fibrations. Then, we expect that 
the direct sum of ADE lattices 
which represent the Kodaira type of reducible 
degenerations of fibers, 
coincides with $\Lambda_{ABE}(\X^{\vee})$. 
Indeed, in every $2d\le 4$ case, 
they coincide by the calculation of \cite[\S7]{Dolg}. 
\end{Rem}

We conclude the section by making an easy but important remark. 

\begin{Prop}[Denseness of algebraic limits]
Note that for each $d\ge 1$, we can consider 
$\overline{\mathcal{F}_{2d}}\to  
\overline{\mathcal{M}_{K3}}^{{\rm Sat},\tau}$ (see Lemma 
\ref{II.lem}, also \S \ref{Alg.limits.sec}). 
If we consider the union of such limits: 
$$\bigcup _{d\in \Z_{>0}}(\partial \overline{\mathcal{F}_{2d}}\cap 
\mathcal{M}_{K3}(d)^{\tau}),$$ 
then this countable set is dense inside the whole 
$17$-dimensional strata $\mathcal{M}_{\rm K3}(d)^{\tau}$. 
\end{Prop}
\begin{proof}
This easily follows since 
$\mathcal{M}_{\rm K3}(d)^{\tau}$ is the quotient of 
$$\{\lambda\in \Lambda_{\rm seg}\otimes \R\mid \lambda^{2}>0\},$$
while $\Lambda_{\rm seg}$ is an even {\it integral} lattice. 
\end{proof}

This implies the following straightforwardly. 
\begin{Cor}[Possible PL invariants for type II degenerations]
Possible PL invariants for type II degenerations 
of polarized K3 surfaces run over a dense subset of 
which appears in \cite[\S 7A]{ABE} and \cite{Osh}. 
\end{Cor}

This result in particular gives {\it negative} answer to 
the first question of \cite[\S 2.6]{HSZ} on the 
behaviour of $V$. 


\section{\cite{HSVZ} glued metric and 
Type II limits of algebraic K3 surfaces}\label{HSVZ.sec}

The recent work of Hein-Sun-Viaclovsky-Zhang \cite{HSVZ} gives 
construction of compact K3 surfaces at the level of hyperK\"ahler structures, 
by glueings of Tian-Yau metrics and Taub-NUT type metrics, 
which maps and collapses to an interval. 

In this section, we reveal how \cite{HSVZ} fits into our picture, 
therefore giving more structures. 
As a result, {\it loc.cit} 
roughly corresponds to two following aspects simultaneously: 
\begin{Aspect}\label{Aspect.EAE}
 the special stable type 
$\E\A\E$ in \cite{ABE} (cf., also our \S \ref{ABE.sec},  \S \ref{measure.decide}), 
\end{Aspect}
\begin{Aspect}\label{Aspect.two.Lag}
also the pushforward of two Lagrangian fibrations on the 
limitting K3 surfaces. 
\end{Aspect} 

\subsection{Review of \cite{HSVZ} construction}

First, we recall their construction here (while we leave full details 
to {\it loc.cit}). 
They construct compact hyperK\"ahler manifolds (hence homeomorphic to  
the K3 surfaces) by glueing, which maps to an interval, 
from the following set of data: 

\begin{itemize}
\item two arbitrary DelPezzo surfaces $X_{1}$ with 
the degrees 
$d_{1}:=(-K_{X_{1}})^{2}$ and $d_{2}:=(-K_{X_{2}})^{2}$, 
\item choice of their (isomorphic) smooth anticanonical divisors $D_{i}\subset X_{i} (i=1,2)$ 
with an isomorphism $D_{1}\simeq D_{2}$, 
\item Tian-Yau metrics (\cite{TY90}) on $X_{i}\setminus D_{i}$ 
(note $\chi(X_{i}\setminus D_{i})=12-d_{i}$) 
which is cohomologically zero in $H^{2}(X_{i}\setminus D_{i},\R)$, 
\item a transition region $\mathcal{N}$ whose general fibers 
over the interval are 
$(T^{2}\times \mathbb{R})$  away from 
$(d_{1}+d_{2})\text{-points}$ in the base, 
\item a hyperK\"ahler metric on $\mathcal{N}$ constructed by 
the Gibbons-Hawking ansatz, 
\item (parameter specifying the attaching parameter 
for the $S^{1}$-rotation), 
\item the ``collapsing parameter'' $\beta\in (0,1]$. 
\end{itemize}

As for the Tian-Yau metric of above situation, 
they analyzed its asymptotic at the boundary $D_{i}$ to 
identify with  
ALH (or ALG$^{*}$ suggested by \cite{CC}) 
with exactly quadratic curvature decay and 
the non-integer 
volume growth $\sim r^{\frac{4}{3}}$ (cf., also 
\cite[Theorem 1.5(iii), $I_{b}$-case]{Hein}), 
where $r$ denotes the distance from some arbitrary base point. 

From the above data, {\it loc.cit} glues the Tian-Yau hyperK\"ahler 
metrics on $(X_{i}\setminus D_{i})$ 
and 
some Gibbons-Hawking metrics with several 
(multi-)Taub-NUT asymptotics 
on $\mathcal{N}$, 
which collapses to the interval $[0,1]$ when $\beta\to 0$ 
(also see earlier expectation by R.Kobayashi \cite[p223]{Kobb90}), 
which we here write $S_{\beta}$ with its 
hyperK\"ahler metric $g_{\beta}$. 
Furthermore, they provide a continuous map 
$F_{\beta}\colon S_{\beta}\to [0,1]$ which satisfies: 
\begin{enumerate}
\item the fibers over ends $F_{\beta}^{-1}(0)$ and $F_{\beta}^{-1}(1)$ 
are closure of open locus in the Tian-Yau spaces $X_{i}\setminus D_{i}$, 
\item \label{aff.str}
for $\beta\to 0$, 
$(S_{\beta},g_{\beta})$ converges in the Gromov-Hausdorff sense 
to the unit interval 
with natural affine structure (induced from the behaviour of 
harmonic functions on $S_{\beta}$), 
\footnote{We observe in general this affine structure is not 
same as the one 
induced from non-archimedean structure as used in \cite{BJ}.} 
\item \label{V.HSVZ}
the limit measure on the interval is written as 
${\sqrt{V(x)}}dx $ with a convex PL function on $[0,1]$ 
with $V(0)=V(1)=0$, 
where $dx$ stands for the affine structure above. 
\end{enumerate}

\begin{Rem}
With respect to this affine structure $dx$, assuming the Gromov-Hausdorff limit of rescaled metrics with fixed diameters is identified with 
the dual graph, the natural affine structure with respect to the latter perspective is $V(x)dx$ (see \cite{BJ}). 
\end{Rem}

Recall that \cite{TY90} first constructed 
the hermitian metric $h$ on the normal bundle for 
$D_{i}\subset X_{i}$ whose curvature form is Ricci-flat, 
then solved the complex Monge-Amper\'e equation with 
the reference metric of Calabi-ansatz type via $h$. 

As \cite{FukI,FukII,CFG} show, the fibers are infranilmanifolds, 
indeed simply Heisenberg nilmanifolds (cf., also \cite[\S 2.2]{HSZ}). 
In particular, they also 
confirmed their hyperK\"ahler manifolds are parametrized by 
$57$-dimensional data (plus rescaling data), i.e., at least containing 
some open subset of $\mathcal{M}_{\rm K3}$. 
Note that $V(0)=V(1)=0$ condition of the above \eqref{V.HSVZ} 
infers, as \cite{Osh} shows logically, 
it should only gives a neighborhood of $E(A)E$ type subcone of 
the fundamental polygon $P (\simeq {\mathcal{M}}_{\rm K3}(d)^{\tau})$ 
in the whole ${\mathcal{M}}_{\rm K3}^{\rm Sat,\tau}$, 
hence the direction which involves D type is missing. 

\vspace{4mm}

Then after \cite{HSVZ}, 
more recent work of Honda-Sun-Zhang \cite{HSZ} proved 
similar PL structure for all possible {\it limit measure} on the 
Gromov-Hausdorff limit when it is $1$-dimensional (interval). 
In \S 2.6 of {\it loc.cit}, they raise some questions regarding the 
function $V$ to which we answer: 

\begin{itemize}

\item First question in {\it loc.cit} asks 
if $V(p)=0$ at the boundary point $p$ in the case when 
$V$ is not constant. This is far from true, from the presence of 
D type region combined with Theorem \ref{continuity2}. 

\item The second question in {\it loc.cit}, in the situation of 
\cite{HSVZ}, asks 
if $V$ is singular at the $d_{1}+d_{2}$ points in the interval. 
The answer is yes from our conclusion. 

\item Their third question is about the ratio of slopes. 
As our analysis so far, the slopes can be normalized to 
$0, \pm 1, \cdots, \pm 9$ and the ratios are rational as expected. 

\end{itemize}

\subsection{Our interpretation of \cite{HSVZ}}

Now, we discuss the aspects \ref{Aspect.EAE}, \ref{Aspect.two.Lag} 
of the beginning of this \S \ref{HSVZ.sec}. 

\subsubsection*{For Aspect\ref{Aspect.EAE} - Landau-Ginzburg model}

Recall that \cite[Theorem 6.4] 
{CJL} relates the above 
Tian-Yau metrics and those of $\frac{4}{3}$-order 
volume growth gravitational instanton on rational 
elliptic surfaces (\cite{Hein}) 
by hyperK\"ahler rotations 
(cf., \cite[6.9]{CJL}, \cite[2.5]{HSVZ}). 
We expect our viewpoint may help to clarify 
relation with the Landau-Ginzburg models \cite{EHX}, 
as we partially give observation here. 

As first instance, we observe that for type II degeneration with 
one component isomorphic to $\mathbb{P}^{2}$, 
the its underlying $\mathbb{R}^{2}$ below our 
degenerate elliptic K3 surface of $X_{3}/\E_{0}$-type (\cite[7.4]{ABE}, \S\ref{ABE.sec},\S\ref{ABE.pre.sec}) 
is 
the limit of the affine structures of Gross-Siebert program 
type at \cite[Example 2.4]{CPS} (see also \cite[\S 3.1]{LLL}), 
which has $3$ $I_{1}$-type singularities of affine structure. 
Indeed, if three of them 
collide via moving worms \cite{KS}, it becomes the abovementioned 
$X_{3}/\E_{0}$-type singularity of affine structures. 
\cite{LLL} also identified it with the affine structure 
coming from special Lagrangian fibration of a complement 
of cubic curve in $\PP^{2}$ constructed in \cite{CJL}. 
See the details at \cite{CPS, CJL, LLL}. 

Also, \cite{ABE} with the arguments in this paper provide further evidence to a variant of 
Doran-Harder-Thompson expectation \cite{Dolg, DHT} for K3 surfaces, 
where ``mirror" is replaced or specialized to be 
hyperK\"ahler rotation, with slight refinement by 
putting $\A$-type surfaces between. 
In particular, this picture applies for general type II degenerations, 
with possibly many irreducible components, hence 
not necessarily Tjurin degeneration in the sense of \cite{DHT}. 

Indeed, recall that in \cite[\S7]{ABE} moduli compactification and 
our reconstruction 
in \eqref{ABE.moduli}, the main role was played by the 
singular fibers behaviour. Such fact 
together with our interpretation of $M_{W}$ as 
limits of hyperK\"ahler rotated K3 surfaces may naturally 
invoke the homological mirror symmetry type phenomenon after  \cite{Sei}, 
that the Lefschetz vanishing cycles around the degenerations 
of the elliptic curves reflect the B-model pictures of the 
degeneration of K3 surfaces. We hope to have 
further understanding of it in our context in 
more systematic way in future.

\subsubsection*{For Aspect \ref{Aspect.two.Lag} - relation with 
two Lagrangian fibrations}

We take a sequence of 
$(g_{8}, g_{12})\in H^{0}(\PP_{s}^{1},\OO(8))\times 
H^{0}(\PP_{s}^{1},\OO(12))$ 
converging to $(3s^{4},s^{6})$ 
and the associated 
Weierstrass K3 surface 
\begin{align*}
\pi''\colon X&:=[y^2 z=4x^3-g_8(s)xz^2+g_{12}(s)z^3]\\
& \subset
\mathbb{P}_{\mathbb{P}_{s}^{1}}(\mathcal{O}_{\mathbb{P}^{1}}(4)\oplus \mathcal{O}_{\mathbb{P}^{1}}(6)\oplus \mathcal{O}_{\mathbb{P}^{1}}) \\ 
&\to B\simeq \PP_{s}^{1}
\end{align*}
converging to $\lambda\in\mathcal{M}_{\rm K3}(d)^{\tau}$ 
in the Satake compactification 
$\overline{\mathcal{M}_{\rm K3}}^{Sat, \tau}$. 
For $i\gg 0$, we have two Lagrangian fibrations: 

\begin{enumerate}
\item \label{e'.ii} 
As we showed in \cite[\S4]{OO18}, for fixed $m\gg 0$, 
we obtain a hyperK\"ahler rotation $X^{\vee}_{m}$ of $X$ 
which is canonically diffeomorphic to $X$ (so that we can 
keep the corresponding marking to original $\varphi$ for $X$) 
whose holomorphic form $\Omega^{\vee}_{m}$ has cohomology class 
as 
\begin{align}
[\Omega^{\vee}_{m}]=|\log \epsilon|^{-1}{\rm Re} \Omega +\sqrt{\frac{-1}{2m}} c(f''+me'').
\end{align} 
Here, $\epsilon$ is as \eqref{e.e'.seq} and 
$c_{i}$ is uniquely determined positive constant 
which automatically converges to $1$ for $i\to \infty$. 
By the same argument as \cite[\S4]{OO18}, we obtain a fibration 
structure $\pi'\colon X^{\vee}_{m}\to \PP_{s}^{1}=B_{m}^{\vee}
$ defined by the pencil $|e'|$ with the fiber class $e'$. 
Note that this is a special Lagrangian fibration with respect to the 
original complex structure, as in \cite[\S 4]{OO18}. 

\item Original $\pi''\colon X\to B\simeq \PP_{s}^{1}$, 
\footnote{Recall that in our first sections, 
the symbol $\pi$ was used as a 
one parameter degeneration of K3 surfaces.}
the Weierstrass elliptic fibration structure. The 
fiber class is $e''$ and is determined as $|e''|$. 
This is Lagrangian fibration with respect to the holomorphic 
volume form $\Omega$. 
\end{enumerate}

As \cite{HSVZ} confirms, its glued K3 surfaces 
form a subset of $\mathcal{M}_{K3}$ which includes 
an open subset $U_{HSVZ}$ whose closure is in 
the EAE region of $\mathcal{M}_{K3}(d)^{Sat, \tau}$. 
We can and do assume that $U_{HSVZ}$ is close to 
the boundary enough so that its any point 
has the special Lagrangian fibration $\pi'$ in \eqref{e'.ii}. 

\begin{Conj}
For any glued fibration of K3 surface to the segment as in \cite{HSVZ} 
so that 
$p=(F_{\beta}\colon X\to [0,1])\in U_{HSVZ}$, 
$F_{\beta}$ factors through both $\pi'$ and $\pi''$. 
There is a $1$-homology class, which we denote $e'\cap e''$, 
such that 
\begin{itemize}
\item 
$e'\cap e''$ is primitive in both $H_{1}(e',\Z)$ and $H_{1}(e'',\Z)$. 
\item 
$e'\cap e''$ is monodromy invariant with respect to both 
$\pi'$ and $\pi''$. 
\end{itemize}
\end{Conj}
The above conjeture would clarify an interpretation of the 
nilmanifold (Heisenberg manifold) fiber of \cite{HSVZ} as 
$S^{1}$-bundle over an elliptic curve. 
\begin{Rem}
It would be interesting to see if the conjectural map 
$B_{m}^{\vee}$ coincides with a moment map for a $\C^{*}$-action 
on it with the McLean metric and the limit measure is comparable to 
its Duistermaat-Heckman measure. 
\end{Rem}

\begin{Rem}
The domain wall crossing \cite[Theorem 1.5]{HSVZ} 
(also treated in Type II superstring theory before 
according to \cite[Remark1.6]{HSVZ}) is now reflected as 
the formation of the singularity of affine structure of 
$I_{w}$ type. 
\end{Rem}


\section{Root lattice type and Type II degenerations} 

Suppose we have a type II polarized degeneration of K3 surfaces $\pi\colon (\X,\mathcal{L})\to \Delta$. 
As an example case, suppose the end component of $\X_0$ is $\F_1$. 
Consider the ample cone of the $\F_{1}$, 
which gives the simplest classical instance of $2$-ray game 
(cf., \cite{Takeuchi} for higher dimensional work) 
of Fano variety: 
Denote the natural projection  $\varphi\colon \F_{1}\to \PP^{2}$, 
$\psi\colon 
\F_{1}\to \PP^{1}$, 
and $H$ the hyperplane in $\PP^{2}$ passing through the 
center of $\varphi$ $p$, 
$E$ the exceptional curve, and set the strict transform of $H$ 
as $H'$ so that $\varphi^{*}H=H'+E$, 
as local notation. 
Then as is well-known and easy, the ample cone is 
$${\rm Amp}(\F_{1})=\R_{\ge 0}[\varphi^{*}H]+\R_{\ge 0}[\pi^{*}
\mathcal{O}_{\PP^{1}}(1)]$$
so that each extremal ray corresponds to $\varphi$ and $\pi$. 

Our point here is that if we consider MMP of $\F_{1}$ 
with scaling in $|L|$, then depending on the terminal objects 
(either $\PP^{1}$ or $\PP^{2}$), we have a subdivision of the cone: 

\begin{align*}
{\rm Amp}(\F_{1})
&=(\R_{\ge 0}[\varphi^{*}H]+\R_{\ge 0}[-K_{\F_{1}}])\\ 
&+(\R_{\ge 0}[-K_{\F_{1}}]+\R_{\ge 0}[\psi^{*}\mathcal{O}_{\PP^{1}}(1)]).
\end{align*}
We denote the first cone as $\mathcal{C}_{1}$ and the second as 
$\mathcal{C}_{2}$. 
Then we observe the following: 
if type II Kulikov degeneration with nef (but generically ample \cite{She}) polarization $\mathcal{L}$ has 
end component $V\simeq \F_{1}$, then 
\begin{itemize}
\item 
$[\mathcal{L}|_{V}]\in \mathcal{C}_{1}$ if and only if it becomes $\E$ type singularity and 
\item 
$[\mathcal{L}|_{V}]\in \mathcal{C}_{2}$ if and only if 
it becomes $\D$ type singularity. 
\end{itemize}
Now we conjecture the following. 
\begin{Conj}[$\D$ vs $\E$ conjecture]
We consider type II polarized degeneration of K3 surfaces $(\mathcal{X},\mathcal{L})\to \Delta$ in $\mathcal{F}_{2d}$. 
Take a simultaneous resolution after base change to make it Kulikov model $\tilde{\mathcal{X}}$. We denote the pull back of $\mathcal{L}$ to $\tilde{\mathcal{X}}$ 
as $\tilde{\mathcal{L}}$ and $\mathcal{X}_0=V_0\cup V_1$, the stable type II degeneration (\cite{Fri84,Kondo}). 

Suppose that if we run the MMP with scaling 
\footnote{As in Example \ref{Ex.Fri}, 
$\tilde{\mathcal{L}}|_{V_{i}}$ can be only semiample and big, 
as the pullback of ample line bundle on some crepant contraction. 
Nevertheless, the MMP with scaling still makes sense. If not 
preferred, one can pass to the crepant contraction induced by 
$\tilde{\mathcal{L}}|_{V_{i}}$ and discuss on it. }
in 
$\tilde{\mathcal{L}}|_{V_i}$ to $V_i$, it ends with ruled surface structure (resp., birational contraction). 
Then our hyperK\"ahler rotation of $(\mathcal{X}_t, \mathcal{L}_t)$ limits to $\D$ type end of interval (resp., $\E$ type end of interval). 
\end{Conj}
\begin{Ex}\label{Ex.Fri}
Indeed, it at least matches to the $4$ cases of degree $2$ examples: see 
\cite[5.2]{Fri84} (cf., also \cite{Shah, AET}). 
\end{Ex}

\begin{Rem}[Strong open K-polystable degenerations on $M_{W}^{\rm nn}$]
For $X:=\mathbb{F}_{2}$, $D$ an elliptic bi-section 
for the ruling, then $X^{o}:=X\setminus D$, 
for certain range of ample $L$, 
$(X^{o},L^{o}:=L|_{X^{o}})$ is strongly 
open K-polystable \cite{ops}, as in the arguments of 
{\it loc.cit}. Indeed, \cite{AP} applied to the 
crepant contraction to the quadric cone 
$X\to \mathbb{P}(1,1,2)$ implies that. 
This appears as $M_{W}^{\rm nn}$ in \cite[\S 7]{OO18}. 
We expect that these $\D$ type degenerating family bubble off 
different ALH gravitational instantons along 
minimal non-collapsing rescaling in the sense of \cite[\S 6]{ops}. 
\end{Rem}



\bigskip

\footnotesize 
\noindent
Email address: \footnotesize 
{\tt yodaka@math.kyoto-u.ac.jp} \\
Affiliation: Department of Mathematics, Kyoto university, Japan  \\


\begin{thebibliography}{FGA}
\tiny{

\bibitem[AV02]{AV}
D.Abramovich, A.Vistoli, 
Compactifying the space of stable maps, 
J.\ Amer.\ Math.\ Soc.\ 15 (1): 27-75. 2002. 

\bibitem[AN99]{AN}
V.~Alexeev, I.~Nakamura,
On Mumford's construction of degenerating abelian varieties, 
Tohoku Math.\ J.\ vol.\ 51, pp.399--420 (1999). 

\bibitem[ABE20]{ABE} 
V.~Alexeev, A.~Brunyate, P.~Engel, 
Compactifications of moduli of elliptic K3 surfaces: stable pairs and 
toroidal, arXiv:2002.07127v3. 



\bibitem[AET19]{AET}
V.~Alexeev, P.~Engel, A.~Thompson, 
Stable pair compactification of moduli of K3 surfaces of degree 2, 
arXiv:1903.09742.

\bibitem[Amb05]{Ambro}
F.~Ambro, 
The moduli b-divisor of an lc-trivial fibration, 
Compositio Math. 141 (2005) 385-403

\bibitem[AP06]{AP}
C. Arezzo, F. Pacard, Blowing up and desingularizing 
K\"ahler orbifolds with constant scalar curvature. Acta Math. 196(2) , 179-228 (2006). 

\bibitem[AMRT]{AMRT}
A. Ash, D. Mumford, M. Rapoport, Y.-S. Tai, Smooth compactifications 
of locally symmetric varieties, Cambridge Mathematical Library, 
Second edition (2010). 

\bibitem[AB19]{AB}
K.Ascher, D.Bejleri, 
Compact moduli spaces of elliptic K3 surfaces, 
arXiv:1902.10686v3. 

\bibitem[AKO06]{AKO}
D.~Auroux, L.~Katzarkov, D.~Orlov, 
Mirror symmetry for Del Pezzo surfaces: Vanishing cycles and coherent sheaves, 
Invent. Math. 166 (2006), no. 3, 537-582. 


\bibitem[BS78]{BS78} 
I.N.Bernstein, O.V.Shvartsman, 
Chevalley's theorem for complex crystallographic Coxeter groups (Russian), 
Func. An. Appl. 12 308 (1978). 


\bibitem[BJ17]{BJ}
S.~Boucksom, M.~Jonsson, 
Tropical and non-Archimedean limits of degenerating families of volume forms, 
Journal de l'\'Ecole polytechnique - Math\'ematiques,  Tome 4  (2017),  p. 87-139. 


\bibitem[BL00]{BL}
J.~Bryan, N.C.Leung, 
The enumerative geometry of K3 surfaces and modular forms, 
J.\ Amer.\ Math.\ Soc.\ 13 (2000), no.2, 371-410. 

\bibitem[Brun15]{Brun}
A.~Brunyate, 
A modular compactification of the space of elliptic K3 surfaces, 
UGA Ph.D thesis (2015). 

\bibitem[CPS]{CPS}
M.Carl, M.Pumperla, B.Siebert, 
A tropical view on Landau-Ginzburg models, 
preprint. 

\bibitem[CFG92]{CFG}
J. Cheeger, K. Fukaya, and M. Gromov, Nilpotent structures and invariant metrics on collapsed manifolds, 
J. Amer. Math. Soc. 5 (1992), 327-372.  


\bibitem[CC97]{ChCo}
J.~Cheeger, T.H.~Colding, 
On the structure of spaces with Ricci curvature bounded below I, 
J.\ Differential Geom.\ 46 (1997), no.~3, 406--480. 

\bibitem[CC16]{CC}
G.~Chen, X.-X.~Chen,  
Gravitational instantons with faster than quadratic curvature decay (III), arXiv:1603.08465. 

\bibitem[CJL19]{CJL}
T.Collins, A.Jacob, Y-S.Lin, 
Special Lagrangian submanifolds of log Calabi-Yau manifolds, 
arXiv:1904.08363 (2019). 

\bibitem[CD07]{CD}
A. Clingher, C. Doran, 
Modular invariants for lattice polarized K3 surfaces, 
Michigan Math. J. 55 (2007), no. 2, 355-393. 

\bibitem[CM05]{CM05}
A. Clingher, J. Morgan, 
Mathematics underlying the F-theory/heterotic string duality in eight dimensions, 
Communications in Mathematical Physics 254 (3), 513-563 (2005). 

\bibitem[Cox95]{Cox}
D.~Cox, 
The homogenous coordinate ring of a toric variety, 
J. Algebraic Geom., 4 (1995): 17-50. 

\bibitem[Dav65]{Dav}
H.~Davenport, 
On $f^{3}(t)-g^{2}(t)$, 
Norske Vid.\ Slesk. Forh.\ (Trondheim) 38 (1965), 86-87. 

\bibitem[DHT17]{DHT}
C.F.Doran, A.Harder, A.Thompson, 
Mirror symmetry, Tyurin degenerations and fibrations on 
Calabi-Yau manifolds, proceedings of the conference String-Math, 2015, 
93-131 (2017). 

\bibitem[Dol96]{Dolg}
I.Dolgachev, 
Mirror symmetry for lattice polarised K3 surfaces, 
J. Math. Sci. 81 (1996), no.3, 2599-2630. 


\bibitem[Don02]{Don02}
S.~Donaldson, 
Scalar curvature and stability of toric varieties, 
J.\ Differential Geom.\ 62 (2002), no.~2, 289--349.

\bibitem[Dre]{Drezet}
J-M. Dr\'ezet, Luna's slice theorem and applications, 
Algebraic group actions 
and quotients, 39-89, Hindawi Publ. Corp., Cairo, (2004).

\bibitem[EF19]{EF}
P.Engel, R.Friedman, 
Smoothings and rational double point adjacencies for 
cusp singularities, J. Differential Geom. (2019). 

\bibitem[EHX97]{EHX}
T.Eguchi, K.Hori, C-S. Xiong, 
Gravitational quantum cohomology, 
Internat. J. Modern Phys. A 12 (1997), no.9, 1743-1782. 

\bibitem[Fjn18]{Fjn}
O.Fujino, 
Semipositivity theorems for moduli problems, Ann.\ of Math., 
pp. 639-665 from Volume 187 (2018),

\bibitem[Fri84]{Fri84}
R.~Friedman, 
A new proof of the global Torelli theorem for K3 surfaces, 
Ann.\ of Math. vol.\ 120, no.2, 237-269 (1984). 

\bibitem[FMW97]{FMW}
R.Friedman, J.Morgan, E.Witten, 
Vector bundles and F-theory, 
Commun. Math. Phys. 187. 679-743 (1997). 


\bibitem[Fscl16]{Fos}
L.~Foscolo, 
ALF gravitational instantons and collapsing Ricci-flat metrics on the K3 surface,
to appear in J.~Differential Geom. (arXiv:1603.06315).


\bibitem[Freed99]{Freed99}
D.~Freed, Special K\"ahler manifolds, 
Comm. Math. Phys. 203 (1999), no. 1, 31--52. 

\bibitem[Fri83]{Fri.smoothing}
R.~Friedman, 
Global smoothings of varieties with normal crossings, 
Annals of Mathematics, 118 (1983), 75-114

\bibitem[Fri84]{Fri}
R.~Friedman, 
A new proof of the global Torelli theorem for K3 surfaces, 
Ann. of Math. (2) 120 (1984), no. 2, 237--269. 

\bibitem[FriMrg94]{FMg}
R.~Friedman and J.~Morgan, Smooth Four-Manifolds and Complex Surfaces,  
Ergebnisse der Mathematik und ihrer Grenzgebiete (3), 27. Springer-Verlag, (1994). 

\bibitem[FriMor83]{FM}
R.~Friedman, D.~Morrison, 
The birational geometry of degenerations,
Progress in Math.\ 29, Birkh\"auser, (1983). 

\bibitem[FriSca86]{FriS}
R.~Friedman, F.~Scattone, 
Type III degenerations of $K3$ surfaces, Invent.\ Math.\ 83 (1986), no.~1, 1--39. 

\bibitem[Fri15]{Fri.recent}
R.~Friedman, 
On the geometry of anticanonical pairs, 
arXiv:1502.02560. 

\bibitem[Fuk87a]{Fuk.mGH}
K. Fukaya, 
Collapsing of Riemannian manifolds and 
eigenvalues of Laplace operator, Invent.\ Math., 
87 (1987), 517-547. 

\bibitem[Fuk87b]{FukI}
K. Fukaya, Collapsing Riemannian manifolds to ones of lower dimensions 
J. Differential Geom. 25 (1987), 139-156. 

\bibitem[Fuk89]{FukII}
K. Fukaya, Collapsing Riemannian manifolds to ones of lower 
dimensions, II, 
J. Math. Soc. Japan 41 (1989), no. 2, 333--356 

\bibitem[GHK15]{GHK}
M. Gross, P.Hacking, S.Keel, 
Moduli of surfaces with an anti-canonical cycle, 
Comp.\ Math Volume 151, Issue 2 (2015) , pp. 265-291. 

\bibitem[GTZ13]{GTZ1}
M.~Gross, V.~Tosatti, Y.~Zhang, 
Collapsing of abelian fibered Calabi-Yau manifolds, Duke Math.\ J.,
162, (2013), no.~3, 517--551. 

\bibitem[GTZ16]{GTZ2}
M.~Gross, V.~Tosatti, Y.~Zhang, 
Gromov-Hausdorff collapsing of Calabi-Yau manifolds, Comm.\ Anal.\ Geom.\ 24 (2016), no.~1, 93--113. 

\bibitem[Hart]{Hart}
R.~Hartshorne, 
Algebraic Geometry, Graduate Texts in Mathematics, Springer-Verlag. 

\bibitem[HKY]{HKY}
P.Hacking, S.Keel, T.Y.Yu, 
Theta functions and the secondary fan for Fano varieties, 
in preparation. 

\bibitem[HL]{HL}
G.~Heckman, E.~Looijenga, The moduli space of rational elliptic surfaces, Algebraic geometry 2000, Azumino (Hotaka), Adv. Stud. Pure Math., vol. 36, Math. Soc. Japan, Tokyo, 2002, pp. 185-248. 

\bibitem[HaZu94]{HZII}
M.~Harris, S.~Zucker, 
Boundary cohomology of Shimura varieties II. Hodge theory at the boundary, 
Invent.\ Math.\ 116 (1994), 243--307. 


\bibitem[HaUe18]{HU}
K.Hashimoto, K.Ueda, 
Reconstruction of general elliptic K3 surfaces from their Gromov-Hausdorff limits, arXiv:1805.01719



\bibitem[Hein12]{Hein}
H-J. Hein, 
Gravitational instantons from rational elliptic surfaces, 
J. Amer. Math. Soc. 25 (2012), 355-393


\bibitem[HSVZ18]{HSVZ}
H.-J.~Hein, S.~Sun, J.~Viaclovsky, R.~Zhang, 
Nilpotent structures and collapsing Ricci-flat metrics on K3 surfaces, 
arXiv:1807.09367. 

\bibitem[HeTos15]{HT}
H.-J.~Hein, V.~Tosatti, 
Remarks on the collapsing of torus fibered Calabi-Yau manifolds, 
Bull. Lond.\ Math.\ Soc.\ 47 (2015), no.~6, 1021--1027.

\bibitem[HSZ19]{HSZ}
S.~Honda, S.~Sun, R.~Zhang, 
A note on the collapsing geometry of hyperK\"ahler four manifolds, 
Sci. China Math. 62, 2195-2210 (2019).  

\bibitem[Huy99]{Huy.HK}
D.~Huybrechts, 
Compact hyper-K\"ahler manifolds: basic results, 
Invent.\ Math.\ 135 (1999), no.~1, 63--113.


\bibitem[Huy01]{Huy.HK.book}
D.~Huybrechts, 
Compact HyperK\"ahler manifolds, in 
``Calabi-Yau Manifolds and Related Geometries''
Lectures at a Summer School in Nordfjordeid, Norway, June 2001. 
Springer-Verlag (2001). 


\bibitem[Huy04]{HuyMM}
D. Huybrechts, 
Moduli spaces of hyperk\"ahler manifolds and mirror symmetry. In \textit{Intersection theory and moduli}, 185--247, ICTP Lect.\ Notes, XIX, Abdus Salam Int.\ Cent.\ Theoret.\ Phys., Trieste, (2004).

\bibitem[Huy16]{Huy}
D.~Huybrechts, 
Lectures on K3 surfaces,  Cambridge Studies in Advanced Mathematics, 158, 
Cambridge University Press (2016). 

 

\bibitem[Iit82]{Iit82}
S. Iitaka,  Algebraic Geometry -- 
An Introduction to Birational Geometry of Algebraic Varieties.
Graduate Texts in Mathematics 76. Berlin: Springer, 1982.

\bibitem[Kas77]{Kas}
A. Kas, Weierstrass normal forms and invariants of elliptic surfaces, Trans. Amer. Math. Soc, 225 (1977), 259--266. 

\bibitem[KeMo97]{KeelMori}
S.Keel, S.Mori, 
Quotients by groupoids, Annals of Mathematics (1997), 2, 145 (1): 193-213. 

\bibitem[Kob90a]{Kob90}
R.~Kobayashi, 
Moduli of Einstein metrics on a $K3$ Surface and degeneration of type I, 
In \textit{K\"ahler Metrics and Moduli Spaces}, 257--311, 
Adv.\ Stud.\ Pure Math., 18-II, T.~Ochiai. ed.\ Academic Press, (1990). 

\bibitem[Kob90b]{Kobb90}
R.~Kobayashi, 
Ricci-flat K\"ahler metrics on affine algebraic manifolds 
and degenerations of K\"ahler-Einstein K3 surfaces, 
In \textit{K\"ahler Metrics and Moduli Spaces}, 257--311, 
Adv.\ Stud.\ Pure Math., 18-II, T.~Ochiai. ed.\ Academic Press, (1990). 


\bibitem[Kod63]{Kod}
K.~Kodaira, 
On compact analytic surfaces: II, 
Ann.\ of Math.\ 77 (1963), no.~3, 563--626. 




\bibitem[KolMor98]{KM}
J.~Koll\'ar, S.~Mori, 
Birational Geometry of Algebraic Varieties, 
Cambridge Tracts in Mathematics, vol.\ 134, 
Cambridge University Press, (1998). 


\bibitem[KS06]{KS}
M.~Kontsevich, Y.~Soibelman, 
Affine structures and non-archimedean analytic spaces, 
In \textit{The Unity of Mathematics}, 321--385,
Progr.\ Math., 244, Birkh\"auser, 2006. 


\bibitem[Kon85]{Kondo}
S.~Kondo, 
Type II degeneration of K3 surfaces, 
Nagoya J. Math (1985) vol.99, 11-30. 

\bibitem[KP17]{Kov}
S. Kov\'acs, Z. Patakfalvi, 
Projectivity of the moduli space of stable log-varieties 
and subadditivity of log-Kodaira dimension, J. Amer. Math. Soc. 30 (2017), no. 4, 
959-1021.

\bibitem[LO19]{LO}
R.Laza, K.OGrady, 
Birational geometry of the moduli space of quartic K3 surfaces, 
Compositio Math. 155 (2019), no. 9, 1655-1710. 

\bibitem[LLL20]{LLL}
S-C.Liu, T-J.Lee, Y-S.Lee, 
On the complex affine structures of SYZ fibration of 
Del Pezzo surfaces, arXiv:2005.04825. 

\bibitem[LM00]{LM}
G.~Laumon, L.~Moret-Bailly, 
Champs Algebriques, 
Ergebnisse der Mathematik, Springer-Verlag Volume 39 (2000). 

\bibitem[Looi76]{Looi76}
E.~Looijenga, Root sytems and elliptic curves, 
Invent.\ Math.\ 38, 17-32 (1976). 

\bibitem[Luna73]{Luna}

D. Luna, Slices \'etales, Sur les groupes alg\'ebriques, Bull. Soc. Math. France
(1973).

\bibitem[Manin]{Manin}
Y. Manin, 
Cubic forms, Algebra, Geometry, Arithmetic,  
Elsevier (1986). 

\bibitem[McL98]{ML}
R. McLean, 
Deformations of calibrated submanifolds, 
Comm.\ Anal.\ Geom.\ 6 (1998), 705--747.




\bibitem[Mil]{Mil}
J.~Milne, 
Algebraic Number theory, Lecture Notes. 
available at \texttt{https://www.jmilne.org/math/CourseNotes/ANT210.pdf}




\bibitem[Mir81]{Mir}
R.~Miranda, The moduli of Weierstrass fibrations over $\mathbb{P}^1$, 
Math.\ Ann.\ 255 (1981), no.~3, 379--394. 


\bibitem[Mum72b]{Mum72.AV}
D.~Mumford, 
An analytic construction of degenerating abelian varieties over complete local rings, 
Compositio Math (1972). 

\bibitem[Mum77]{MumHir}
D.Mumford, 
Hirzebruch's proportionality theorem in the noncompact case, Invent. Math. 42 (1977) 

\bibitem[Nik75]{Nik75}
V.V.~Nikulin, Kummer surfaces, 
Izv.\ Akad.\ Nauk SSSR Ser.\ Mat.\ 39 (1975), no.~2, 278--293, 471. 

\bibitem[Nik79]{Nik}
V.V.~Nikulin, 
Integer symmetric bilinear forms and some of their geometric applications, 
Izv.\ Akad.\ Nauk SSSR Ser.\ Mat.\ 43 (1999), no.~1, 111--177, 238. 

\bibitem[Od18a]{Od.Mg}
Y.~Odaka, 
Tropical Geometric Compactification of Moduli, I, 
- $M_{g}$ case, Moduli of K-stable varieties, Springer INdAM series 31 (2018). 


\bibitem[Od18b]{Od.Ag}
Y.~Odaka, 
Tropical Geometric Compactification of Moduli, II - Ag case and holomorphic limits -
IMRN (2018). 

\bibitem[OO18]{OO18}
Y.~Odaka, Y.~Oshima, 
Collapsing K3 surfaces, Tropical geometry and Moduli compactifications of Satake and Morgan-Shalen type, 
arXiv:1810.07685. 


\bibitem[Odk12a]{Od2}
Y.~Odaka, 
The Calabi conjecture and K-stability, 
Int.\ Math.\ Res.\ Not.\ IMRN (2012), no.~10, 2272--2288. 

\bibitem[Odk12b]{Od12}
Y.~Odaka, 
On the moduli of K\"ahler-Einstein Fano manifolds, 
arXiv:1211.4833v4. Proceeding of Kinosaki algebraic geometry symposium 2013. 

\bibitem[Odk13a]{Od}
Y.~Odaka, 
The GIT stability of polarized varieties via Discrepancy, 
Ann.\ of Math.\ (2) 177 (2013), no.~2, 645--661. 

\bibitem[Odk13b]{Od0}
Y.~Odaka, 
A generalization of the Ross-Thomas slope theory, 
Osaka J.\ Math.\ 50 (2013), no.~1, 171--185. 

\bibitem[Odk14]{TGC.I}
Y.~Odaka, 
Tropical geometric compactification of Moduli, I - $M_g$ case-, 
to appear in Proceeding Volume for the Workshop on moduli of K-stable varieties. 
Springer-Verlag INdAM-series. (available at arXiv:1406.7772)


\bibitem[Odk15]{Od15}
Y.~Odaka, 
Compact moduli spaces of K\"ahler-Einstein Fano varieties, 
Publ.\ Res.\ Int.\ Math.\ Sci.\ 51 (2015), no.~3, 549--565. 


\bibitem[Odk18]{TGC.II}
Y.~Odaka, 
Tropical geometric compactification of Moduli, II - $A_g$ case and holomorphic limits -, 
Int.\ Math.\ Res.\ Not.\ IMRN (2018), 
https://doi.org/10.1093/imrn/rnx293. 

\bibitem[OO18a]{OO.announce}
Y.~Odaka, Y.~Oshima, 
Collapsing K3 surfaces and moduli compactification, 
Proc.\ Japan Acad.\ Ser.\ A Math.\ Sci. 94 (2018), no.~8, 81--86.


\bibitem[OO18b]{OO18}
Y.~Odaka, Y.~Oshima, 
Collapsing K3 surfaces, Tropical geometry and Moduli compactifications of Satake, Morgan-Shalen type, 
arXiv:1810.07685 (2018). 

\bibitem[OSS16]{OSS}
Y.~Odaka, C.~Spotti, S.~Sun, 
Compact moduli spaces of del Pezzo surfaces and K\"ahler-Einstein metrics, 
J. Differential Geom.\ 102 (2016), no.~1, 127--172. 

\bibitem[Od20a]{ops}
Y.Odaka, 
Polystable log 
Calabi-Yau varieties and Gravitational instantons, 
arXiv:2009.13876. 

\bibitem[Od20b]{SBB}
Y.Odaka, 
to appear. 

\bibitem[Od20c]{Od.survey}
Y.Odaka, 
On the K-stability and the moduli problem of varieties, 
Suugaku vol. 72, no.3 2020.  
(in Japanese). 

\bibitem[Osh]{Osh}
Y.~Oshima, 
in preparation. 


\bibitem[Ohno18]{Ohno}
K.~Ohno, 
Minimizing CM degree and slope stability of projective varieties, 
arXiv:1811.12229 (2018). 



\bibitem[Pin77]{Pin77}
H. C. Pinkham, Simple elliptic singularities, Del Pezzo surfaces and Cremona trans- formations, Several complex variables (Proc. Sympos. Pure Math., Vol. XXX, Part 1, Williams Coll., 1975), Amer. Math. Soc., Providence, R. I., 1977, pp. 69-71.

\bibitem[Sei01]{Sei}
P.Seidel, 
More about vanishing cycles and mutation, 
Symplectic Geometry and Mirror Symmetry: 
Proceedings of the $4$-th KIAS annual international 
conference (Seoul, 2000) K. Fukaya, 
Y.-G.Oh, K.Ono, G.Tian, eds. World Sci., 2001, 
429-465. 



\bibitem[Sat56]{Sat0}
I.~Satake, 
On the compactification of the Siegel space, 
J.\ Indian Math.\ Soc.\ 20 (1956), 259--281. 

\bibitem[Sat60a]{Sat1}
I.~Satake, 
On representations and compactifications of symmetric Riemannian spaces, 
Ann.\ of Math.\ (2) 71 (1960), 77--110. 

\bibitem[Sat60b]{Sat2}
I.~Satake, 
On compactifications of the quotient spaces for arithmetically defined discontinuous groups, 
Ann.\ of Math.\ (2) 72 (1960), 555--580. 

\bibitem[Sca87]{Sca}
F.~Scattone, 
On the compactification of moduli spaces for algebraic $K3$ surfaces, 
Mem.\ Amer.\ Math.\ Soc.\ 70 (1987), no.~374.

\bibitem[Schm73]{Schm}
W.~Schmid, 
Variation of Hodge structure: the singularities of the period mapping, 
Invent.\ Math.\ 22 (1973), 211-319. 


\bibitem[Ser73]{Serre}
J.-P.~Serre, 
A course in arithmetic, Graduate Texts in Mathematics, No.~7, 
Springer-Verlag, (1973). 


\bibitem[Sha80]{Shah}
J.~Shah, 
A complete moduli space for K3 surfaces of degree $2$, 
Ann.\ of Math.\ (2) 112 (1980), no. 3, 485--510. 

\bibitem[Sha81]{Shah2}
J.~Shah, 
Degenerations of $K3$ surfaces of degree $4$, 
Trans.\ Amer.\ Math.\ Soc.\ 263 (1981), no. 2, 271--308. 

\bibitem[She83]{She}
N.I.~Shepherd-Barron, Degenerations with numerically effective canonical divisor, 
In \textit{The birational geometry of degenerations (Cambridge, Mass., 1981)}, 33--84,
Progr.\ Math., vol.~29, Birkh\"auser, Boston, (1983).



\bibitem[She83]{She}
N.ShepherdBarron, 
Extending polarizations on families of K3 surfaces, 
(Cambridge, MA, 1981), 
Progr. Math vol 29, 135-171, 
Birkh\"auser, Boston, 1983. 

\bibitem[Shi05]{Shi}
T.Shioda, 
Elliptic surfaces and Davenport-Stothers triples, 
Comment.\ Math.\ Univ.\ St.\ Pauli \textbf{54} (2005), 49-68. 

\bibitem[Sto81]{Sto}
W.W.Stothers, 
Polynomial identities and Hauptmoduln, 
Quart.\ J.\ Math.\ Oxford (2) \textbf{32} (1981) 349-370. 

\bibitem[Take89]{Takeuchi}
K. Takeuchi. Some birational maps of Fano 3-folds. Compositio Math., 71(3):265-283, 1989.


\bibitem[TY90]{TY90}
G.~Tian, S-T.Yau, 
Complete K\"ahler manifolds with zero Ricci curvature I, 
J.\ Amer.\ Math.\ Sci., 3 (1990), 579-610. 



\bibitem[YZ96]{YZ}
S.T.Yau, E.Zaslow, 
BPS states, string duality, and nodal curves on K3, 
Nuclear Phys. B 471 (1996), 503-512.   

\bibitem[Zan95]{Zan}
U.~Zannier, 
On Davenport's bound for the degree of $f^{3}-g^{2}$ and 
Riemann's Existence theorem, 
Acta Arithmetica LXXI.2 (1995). 
}


\end{thebibliography}
\end{document}